\documentclass[hidelinks]{tran-l}

\usepackage{amsmath,amssymb,eucal,graphicx,mathtools}
\usepackage{amsthm}
\usepackage{hyperref}
\usepackage{graphicx}
\usepackage{caption}
\usepackage{subcaption}
\usepackage{listings}
\usepackage{tikz-cd}
\usepackage{tikz}
\usetikzlibrary{shapes.geometric, positioning}
\usetikzlibrary{arrows}
\usetikzlibrary{decorations.pathmorphing}
\usepackage{stmaryrd}
\usepackage[capitalise]{cleveref}
\usepackage{comment}
\usepackage{verbatim}
\usepackage{tkz-euclide}
\usepackage{adjustbox}
\usepackage{mathdots}
\usepackage{circuitikz}
\usetikzlibrary{decorations.markings}
\usepackage{xcolor}
\hypersetup{
    colorlinks,
    linkcolor={red!75!black},
    citecolor={blue!70!black},
    urlcolor={blue!80!black}
}

\usepackage{geometry}
\newcommand{\R}{\mathbb{R}}

\newcommand{\Acal}{\mathcal{A}}
\newcommand{\Bcal}{\mathcal{B}}
\newcommand{\Ccal}{\mathcal{C}}
\newcommand{\Dcal}{\mathcal{D}}

\newcommand{\Scal}{\mathcal{S}}
\newcommand{\Tcal}{\mathcal{T}}

\newcommand{\Wcal}{\mathcal{W}}
\newcommand{\Ucal}{\mathcal{U}}
\newcommand{\Vcal}{\mathcal{V}}
\newcommand{\Xcal}{\mathcal{X}}
\newcommand{\Ycal}{\mathcal{Y}}

\newcommand{\Cfrak}{\mathfrak{C}}

\newcommand{\Tfrak}{\mathfrak{T}}

\newcommand{\Wfrak}{\mathfrak{W}}
\newcommand{\ifrak}{\mathfrak{i}}

\newcommand{\Trm}{\mathrm{T}}

\DeclareMathOperator{\Hom}{Hom}

\DeclareMathOperator{\mods}{mod}

\DeclareMathOperator{\End}{End}

\DeclareMathOperator{\Ext}{Ext}

\DeclareMathOperator{\image}{im}

\DeclareMathOperator{\spann}{span}
\DeclareMathOperator{\coker}{coker}

\DeclareMathOperator{\add}{add}
\DeclareMathOperator{\proj}{proj}

\DeclareMathOperator{\Filt}{Filt}
\DeclareMathOperator{\Fac}{Fac}

\DeclareMathOperator{\ttilt}{\tau-tilt}

\DeclareMathOperator{\Hasse}{Hasse}
\DeclareMathOperator{\ftors}{f-tors}

\DeclareMathOperator{\tors}{tors}
\DeclareMathOperator{\brick}{brick}
\DeclareMathOperator{\jirrc}{j-irr^c}

\DeclareMathOperator{\itv}{itv}

\DeclareMathOperator{\wide}{wide}

\DeclareMathOperator{\sbrick}{sbrick}

\DeclareMathOperator{\ideal}{ideal}

\DeclareMathOperator{\tauint}{\tau-itv}
\DeclareMathOperator{\simp}{simp}

\copyrightinfo{}{}

\newtheorem{thm}{Theorem}[section]
\newtheorem{lem}[thm]{Lemma}
\newtheorem{cor}[thm]{Corollary}
\newtheorem{prop}[thm]{Proposition}

\theoremstyle{definition}
\newtheorem{defn}[thm]{Definition}
\newtheorem{exmp}[thm]{Example}

\theoremstyle{remark}
\newtheorem{rmk}[thm]{Remark}

\numberwithin{equation}{section}

\begin{document}

\title{$\tau$-cluster morphism categories of factor algebras}



\author{Maximilian Kaipel}
\address{Abteilung Mathematik, Department Mathematik/Informatik der Universität
zu Köln, Weyertal 86-90, 50931 Cologne, Germany}
\curraddr{}
\email{mkaipel@uni-koeln.de}
\thanks{}


\date{}

\dedicatory{}

\begin{abstract}
    We take a novel lattice-theoretic approach to the $\tau$-cluster morphism category $\Tfrak(A)$ of a finite-dimensional algebra $A$ and define the category via the lattice of torsion classes $\tors A$. Using the lattice congruence induced by an ideal $I$ of $A$ we establish a functor $F_I: \Tfrak(A) \to \Tfrak(A/I)$. If $\tors A$ is finite, $F_I$ is a regular epimorphism in the category of small categories and we characterise when $F_I$ is full and faithful. The construction is purely combinatorial, meaning that the lattice of torsion classes determines the $\tau$-cluster morphism category up to equivalence. 
\end{abstract}

\maketitle

\section{Introduction}

Torsion classes are an important class of subcategories of an abelian category axiomatising the properties of torsion groups in the category of abelian groups \cite{Dickson66}. They are closely related to $t$-structures of triangulated categories \cite{BBD} and have analogues in a plethora of categorical setting \cite{AdachiEnomotoTsukamoto,BeligiannisReiten2007,Jorgensen16,Tattar2021}. In representation theory, torsion classes are essential objects of tilting theory \cite{BrennerButler}, $\tau$-tilting theory \cite{AIR2014} and Auslander-Reiten theory \cite{AuslanderSmalo1980}. In combinatorics, torsion classes arise as the Tamari lattice \cite{Thomas2012}, as Cambrian lattices \cite{IngallsThomas2009} and as the weak order of Weyl groups \cite{Mizuno2013}. Their purely lattice-theoretic properties have also inspired substantial research \cite{BarnardCarrolZhu19, DIRRT2017, IRRT2018,GarverMcConville19, Ringel2018}. Additionally, the study of torsion classes is motivated by their connection to cluster algebras \cite{BrustleYang}. \\

Wide subcategories of abelian categories were first considered in \cite{Hovey2001} and are intimately related to torsion classes \cite{MarksStovicek} and the study of stability conditions \cite{Asai2019WS,BST2019,King1994,Yurikusa2018}. Combinatorially, wide subcategories arise as non-crossing partitions \cite{IngallsThomas2009} and as the shard-intersection order of Weyl groups \cite{Thomas2018}. In fact, the poset of wide subcategories is encoded combinatorially in the lattice of torsion classes \cite{Enomoto2023}. We are particularly interested in the \textit{$\tau$-perpendicular wide subcategories} \cite{Jasso2015}, which generalise classical perpendicular categories \cite{GeigleLenzing1991}. These wide subcategories are equivalent to module categories and they have been the focus of much research \cite{AsaiPfeifer2019, BuanHanson2017, BuanMarsh2018, BuanMarsh2021, DIRRT2017}.\\

In a sequence of papers \cite{IgusaTodorov2017, BuanMarsh2018, BuanHanson2017} the \textit{$\tau$-cluster morphism} category $\Wfrak(A)$ was introduced for a finite-dimensional algebra $A$. Its objects are the $\tau$-perpendicular wide subcategories and factorisations of its morphisms are given by generalised exceptional sequences \cite{BuanMarsh2021}. One purpose of the $\tau$-cluster morphism category is to study the \textit{picture group} $G(A)$, which encodes the covering relations of the lattice of torsion classes \cite{IgusaTodorovWeyman2016, HansonIgusa2021}. The picture group is the fundamental group of the classifying space of $\Wfrak(A)$ and in many cases a $K(\pi,1)$ space \cite{BarnardHanson2022,HansonIgusa2021, HansonIgusaPW2SMC, IgusaTodorov2017, IgusaTodorov22, Kaipelcatpartfan}. Moreover, the picture group is closely related to maximal green sequences \cite{Keller2011,IgusaTodorovMGS2021} arising in the context of Donaldson-Thomas invariants and BPS states in physics \cite{KS2008}. \\

Beside the original algebraic approach, the geometric construction of an equivalent category $\Cfrak(A) \cong \Wfrak(A)$ from the $g$-vector fan \cite{DIJ2019} achieved in \cite{STTW2023} has enabled new insights into the $\tau$-cluster morphism category \cite{Kaipelcatpartfan}. On the other hand, the generalisation of $\Wfrak(A)$ to certain differential graded algebras and differential graded categories \cite{Borve2022, borve2024silt} aims to connect $\Wfrak(A)$ with the pro-unipotent group of the motivic Hall algebra considered in \cite{Bridgeland2017}. In this article, we begin by constructing an equivalent category $\Tfrak(A) \cong \Cfrak(A) \cong \Wfrak(A)$ from the lattice of torsion classes. This construction relates to that from the $g$-vector fan by \cite{STTW2023} via the bijection of maximal $g$-vector cones with functorially-finite torsion classes \cite{AIR2014,DIJ2019}. However, this definition enables a lattice-theoretic approach which yields new insights into the structure of the $\tau$-cluster morphism category.

\begin{thm}\label{thm:combintrothm}
    Let $A$ be a finite-dimensional algebra, then we may define the $\tau$-cluster morphism category $\Tfrak(A)$ from the lattice of torsion classes $\tors A$. Moreover, if $\tors A$ is finite, the category is determined entirely by the underlying abstract lattice structure of $\tors A$.
\end{thm}
\begin{proof}
    The definition of $\Tfrak(A)$ from $\tors A$ is \cref{defn:latticedef}, the equivalence with previous constructions is \cref{thm:welldefn} and the combinatorial construction from the underlying lattice structure is \cref{thm:categoryiso}.
\end{proof}

This result fits into a larger body of research aiming to recover algebraic objects and structures related to the category finite-dimensional (right) $A$-modules, $\mods A$, via the poset structure of its torsion classes, see for example \cite{BarnardHanson2022exc, DIRRT2017, Enomoto2023, kase2017}. As a consequence of this construction, we obtain the following result for signed ($\tau$-)exceptional sequences, as introduced in \cite{BuanMarsh2021, IgusaTodorov2017}, which are factorisations of morphisms in $\Tfrak(A)$ into irreducible ones \cite[Thm. 11.8]{BuanMarsh2018}.

\begin{cor}\label{cor:introcor1.2}
    Let $A$ and $B$ be finite-dimensional algebras such that  $\tors A \cong \tors B$ and both are finite lattices. Then $\Tfrak(A) \cong \Tfrak(B)$ is an equivalence of categories and there is a bijection
    \[ \{ \text{signed $\tau$-exceptional sequences of $\mods A$}\} \longleftrightarrow \{ \text{signed $\tau$-exceptional sequences of $\mods B$}\}.\]
\end{cor}
\begin{proof}
    The equivalence of categories is \cref{cor:equivcats}, and the bijection of signed $\tau$-exceptional sequences is \cref{cor:signedexcep}.
\end{proof}

For example, an algebra $A$ and its quotient $A/\langle c \rangle$ satisfy $\tors A \cong \tors A/\langle c \rangle$ when $c$ is a central element contained in the Jacobson radical by \cite{EJR2018}. This result complements the bijection between (unsigned) $\tau$-exceptional sequences of \cite{BarnardHanson2022exc} for algebras $A$ and $B$ such that $\tors A \cong \tors B$ and both are finite lattices, using a combinatorial description of such sequences as $\kappa^d$-exceptional sequences. The combination of \cite[Thm. 8.10]{BarnardHanson2022exc} and \cref{cor:introcor1.2} suggests that the recently defined mutation of $\tau$-exceptional sequences \cite{BHM2024} may be encoded in the lattice of torsion classes as well. \\

To show that the classifying space $\Bcal \Tfrak(A)$ of the $\tau$-cluster morphism category is a $K(\pi,1)$ space, it is sufficient is to find a faithful (group) functor to some groupoid with one object and study a combinatorial (pairwise compatibility) condition \cite{Igusa2022}. This is the standard approach taken in \cite{BarnardHanson2022,HansonIgusa2021, HansonIgusaPW2SMC, IgusaTodorov2017, IgusaTodorov22, Kaipelcatpartfan}. As another consequence of \cref{cor:introcor1.2} we can extend the existence of a faithful group functor from $\Tfrak(A)$ to $\Tfrak(B)$ for algebras $B$ whose lattice of torsion classes is finite and isomorphic to that of $A$. Such a functor is known to exist for hereditary algebras \cite{IgusaTodorov22}, $K$-stone algebras \cite{HansonIgusaPW2SMC} and algebras whose $g$-vector fan is a finite hyperplane arrangement \cite{Kaipelcatpartfan}. \\

Another advantage of defining $\Tfrak(A)$ from the lattice of torsion classes is that it allows us to relate $\Tfrak(A)$ and $\Tfrak(A/I)$ where $I$ is any ideal of $A$. In this situation, the relationship is determined by the lattice congruence on $\tors A$ induced by the ideal $I$, see \cite{DIRRT2017}. Previously, the only connection established between the $\tau$-cluster morphism categories $\Tfrak(A)$ and $\Tfrak(B)$ of two finite-dimensional algebras $A$ and $B$ was in the case where $\mods B$ is equivalent to a wide subcategory of $\mods A$. It was shown in \cite{BarnardHanson2022, BuanHanson2017} that $\Tfrak(B)$ is a full subcategory of $\Tfrak(A)$. Our second main result is the following structural theorem:

\begin{thm} \label{thm:introthm}
    Let $A$ be finite-dimensional and $I$ any ideal of $A$. There exists a functor $F_I: \Tfrak(A) \to \Tfrak(A/I)$.
    If moreover the lattice of torsion classes is finite, then:
    \begin{enumerate}
        \item $F_I$ is surjective-on-objects;
        \item $F_I$ is faithful if and only if $\tors A \cong \tors A/I$;
        \item $F_I$ is full if and only if $\tors A \cong \tors A/I \times \tors B$ for some finite-dimensional algebra $B$;
        \item Every morphism of $\Tfrak(A/I)$ lies in the essential image of $F_I$;
        \item $F_I$ reflects composition of morphisms;
        \item $F_I$ is a regular epimorphism in $\mathcal{C}\mathrm{at}$, the category of small categories;
        \item The classifying space $\Bcal \Tfrak(A/I)$ is a quotient space of the classifying space $\Bcal \Tfrak(A)$;
        \item There exists a surjective group homomorphism $G(A) \to G(A/I)$ between picture groups.
    \end{enumerate}
\end{thm}
\begin{proof}
    The existence of $F_I$ is \cref{thm:inducedfunctor}. If $\tors A$ is finite then (1) is \cref{prop:surjonobj}, (2) is a special case of \cref{lem:FIfaithfuliff}, (3) is \cref{lem:Ffulliff}, (4) is \cref{prop:essentialimage}, (5) is \cref{lem:reflectcomp}, (6) is \cref{cor:regularepi}, (7) is \cref{thm:spacequotient} and (8) is \cref{prop:picgroupquot}.
\end{proof}

The article is organised as follows. In \cref{sec:background} we recall lattice-theoretic definitions of lattices and connect them with $\tau$-tilting theory. We also provide the algebraic and geometric definitions of the $\tau$-cluster morphism category for comparison. In \cref{sec:definition} we give a new definition of the $\tau$-cluster morphism category via $\tau$-perpendicular intervals of the lattice of torsion classes and show that it is a well-defined category and equivalent to the previous constructions. In \cref{sec:tauequiv} we translate the definition into a purely combinatorial setting using join-intervals and completely join-irreducible elements and conclude that the category is determined by the underlying poset whenever it is is finite. We use the new definition in \cref{sec:quotients} to establish a functor $F_I: \Tfrak(A) \to \Tfrak(A/I)$ using the surjective morphism of lattices $- \cap \mods A/I: \tors A \to \tors A/I$ induced by an ideal $I \in \ideal A$. When the lattice is finite, we demonstrate how to lift $\tau$-perpendicular intervals in \cref{sec:epimorphisms} and use this to show that $F_I$ reflects composition of morphisms and that it is a regular epimorphism in the category of small categories. In \cref{sec:classifyingspace} we establish similar quotient relationships for the classifying space and picture group. In the final \cref{sec:examples} we demonstrate the necessity of the finiteness assumption for previous results and compare the definition of $\Tfrak(A)$ with those of $\Wfrak(A)$ and $\Cfrak(A)$ using an example.

\section{Background} \label{sec:background}
Throughout, $K$ is a field and $A$ a finite-dimensional $K$-algebra. The category of finitely-generated (right) modules is denoted by $\mods A$. A \textit{torsion class} is a full subcategory of $\mods A$ which is closed under taking factor modules and extensions \cite{Dickson66}. Dually, a \textit{torsion-free} class is a full subcategory of $\mods A$ which is closed under taking submodules and extensions. A full subcategory of $\mods A$ is called \textit{wide} if it is closed under kernels, cokernels and extensions. Both torsion classes and wide subcategories of $\mods A$ form partially ordered sets (posets) under inclusion and we denote these posets by $\tors A$ and $\wide A$ respectively. We now recall the definition of lattices, for a textbook reference see for example \cite{Gratzer}.

\begin{defn}
    Let $L$ be a poset.
    \begin{enumerate}
        \item $L$ is called a \textit{join-semilattice} if there exists a unique minimal common upper bound, the \textit{join} $x \lor y \in L$ of $x$ and $y$ for all $x,y \in L$.
        \item $L$ is called a \textit{meet-semilattice} if there exists a unique maximal common lower bound, the \textit{meet} $x \land y \in L$ of $x$ and $y$ for all $x,y \in L$.
    \end{enumerate}
    A join-semilattice (resp. meet-semilattice) $L$ is called \textit{complete} if every subset $S = \{x_1, \dots, x_r\} \subseteq L$ admits a unique maximal upper bound $\bigvee S \coloneqq x_1 \lor \dots \lor x_r$ (resp. a unique maximal lower bound $\bigwedge S \coloneqq x_1 \land \dots \land x_r$).
\end{defn}

\begin{defn}
    A poset $L$ which is both a join-semilattice and a meet-semilattice is called a \textit{lattice}. A lattice is \textit{complete} if it is both a complete join-semilattice and a complete meet-semilattice.
\end{defn}

\begin{lem} \cite[Lem. 14]{Gratzer}\label{lem:gratzer}
    The following coincide for a poset $L$:
    \begin{enumerate}
        \item $L$ is a complete join-semilattice.
        \item $L$ is a complete meet-semilattice.
        \item $L$ is complete lattice.
    \end{enumerate}
\end{lem}

This makes the following easy to see.

\begin{exmp}\label{exmp:completelattice}
    The partially ordered sets $\tors A$ and $\wide A$ are complete meet-semilattices with meet given by intersection. By \cref{lem:gratzer} they are therefore complete lattices. For some special properties of the lattice $\tors A$ see for example \cite[Thm. 1.3]{DIRRT2017}
\end{exmp}

\begin{defn}
    Let $L$ be a join-semilattice (resp. meet-semilattice), a \textit{join-subsemilattice} (resp. \textit{meet-subsemilattice}) $K$ is a subset $K \subseteq L$ such that for all $x,y \in K$ we have $x \lor y \in K$ (resp. $x \land y \in K$). A subset $K$ of a lattice $L$ is a \textit{sublattice} if it is both a join-subsemilattice and meet-subsemilattice.
\end{defn}

\begin{defn}
    An equivalence relation $\equiv$ on a complete lattice $L$ is called a \textit{complete lattice congruence} if for an indexing set $I$ and families $\{x_i\}_{i \in I}, \{y_i\}_{i \in I} \subseteq L$ the following holds:
    \[ x_i \equiv y_i \text{ for all } i \in I \quad \Longrightarrow \quad \bigvee \{x_i : i \in I\} \equiv \bigvee \{y_i : i \in I\} \text{ and } \bigwedge \{x_i : i \in I\} \equiv \bigwedge \{y_i : i \in I\} .\]
    Given the corresponding set of partitions $\Phi_{\equiv}$ of an equivalence relation $\equiv$, define the \textit{quotient lattice} $L/\equiv$ as the lattice whose elements are elements of $\Phi$ and such that for $C_1, C_2 \in \Phi$ the element $C_1 \lor C_2$ is the equivalence class $C_3 \in \Phi$ containing $x \lor y$ for some $x \in C_1$ and $y \in C_2$. The meet $C_1 \land C_2$ is defined dually.
\end{defn}

\begin{defn}
    Let $L_1, L_2$ be complete lattices. A map $\eta: L_1 \to L_2$ is a \textit{morphism of complete lattices} if $\eta(\bigvee_{L_1} S) = \bigvee_{L_2} \eta(S)$ and $\eta(\bigwedge_{L_1} S) = \bigwedge_{L_2} \eta(S)$ for all $S \subseteq L_1$.
\end{defn}

The following result will play an essential role throughout this article and act as an example of the notions defined above.
\begin{thm}\cite[Thm. 5.12]{DIRRT2017}\label{thm:DIRRTquot}
    For any ideal $I \in \ideal A$ the map $\overline{(-)}_I: \Tcal \mapsto \overline{(\Tcal)}_I \coloneqq \Tcal \cap \mods A/I$ is a surjective morphism of complete lattices $\tors A \twoheadrightarrow \tors A/I$ and the induced equivalence relation $\Phi_I$ on $\tors A$ is a complete lattice congruence.
\end{thm}

For convenience we write $\overline{(-)} \coloneqq \overline{(-)}_I$ and $\overline{\Tcal} \coloneqq \overline{(\Tcal)}_I$ when no confusion about $I \in \ideal A$ may arise. We can describe lattice congruences in the following way.

\begin{prop} \label{prop:quotlattiso}
    Let $\Phi$ be a complete lattice congruence on $L$, then the equivalence classes $[x] \in \Phi$ are intervals $[\pi_{\downarrow}^{\Phi} x, \pi_{\uparrow}^{\Phi} x] \subseteq L$, where $\pi_{\downarrow}^{\Phi}: L \to \pi_{\downarrow}^\Phi L \coloneqq \image \pi_{\downarrow}^\Phi$ is a morphism of complete lattices, sending elements $x \in L$ to the bottom element in their equivalence class, and $\pi_{\uparrow}^{\Phi}$ is defined dually. Moreover, $\pi_{\downarrow}^{\Phi}(L)$ and $\pi_{\uparrow}^{\Phi}(L)$ are isomorphic to $L/\Phi$.
\end{prop}
\begin{proof}
	If $L$ is a finite lattice, then \cite[Prop. 9-5.2]{Reading2016} implies that an equivalence class $[x] \in \Phi$ corresponds to an interval $[\pi_{\downarrow}^{\Phi} x, \pi_{\uparrow}^{\Phi} x] \subseteq L$ and \cite[Prop. 9-5.5]{Reading2016} states that $\pi_{\downarrow}^{\Phi}: L \to \pi_{\downarrow}^\Phi L$ is a morphism of lattices and that $L/\Phi$ is equivalent to $\pi_{\downarrow}^\Phi L$. The same result holds for $\pi_{\uparrow}^\Phi$ by the dual of \cite[Prop. 9-5.5]{Reading2016}.  By \cite[Exercise 9.42]{Reading2016} these result hold analogously for complete lattices and complete lattice congruences.
\end{proof}

For a complete lattice congruence $\Phi$ on a complete lattice $L$, the surjective morphisms of lattices $L \to L/\Phi$ is sometimes also denoted by $\Phi$. The first part of \cref{prop:quotlattiso} then states that for any $x \in L/\Phi$, the preimage $\Phi^{-1}(x)$ is an interval $[\pi_{\downarrow} \Phi^{-1}(x), \pi_{\uparrow} \Phi^{-1}(x)]$.

\subsection{$\tau$-tilting theory}
Let $\tau$ denote the Auslander-Reiten translation of $\mods A$. In \cite{AIR2014} the authors introduce $\tau$-tilting theory as a simultaneous generalisation of classical tilting theory to higher projective dimensions and a completion of classical tilting theory from the viewpoint of mutation. Denote by $\proj A$ the full subcategory of $\mods A$ consisting of projective modules and for $M \in \mods A$ denote by $|M|$ the number of non-isomorphic indecomposable direct summands of $M$.

\begin{defn}\cite[Def. 0.3]{AIR2014}
    A pair $(M,P) \in \mods A \times \proj A$ is called \textit{$\tau$-rigid} if $\Hom_A(M, \tau M) = 0$ and $\Hom_A(P,M) =0$. It is called $\tau$-tilting if additionally $|M|+ |P|=|A|$.
\end{defn}

For a module $M$ denote by $\Fac M$ the full subcategory of $\mods A$ consisting of factor modules of $M$ and by $M^\perp \coloneqq \{X \in \mods A: \Hom_A(M,X) =0\}$ and dually for ${}^\perp M$. It was shown in \cite[Thm. 5.10]{AuslanderSmalo1980} that $\Fac M$ is a torsion class if $M$ is $\tau$-rigid. Torsion classes arising in this way are called \textit{functorially-finite} and conversely each functorially-finite torsion class arises as $\Fac M$ of some $\tau$-rigid module $M \in \mods A$ by \cite[Prop. 1.1, Prop. 1.2]{AIR2014}. We denote the collection of functorially-finite torsion classes by $\ftors A$. More precisely, by \cite[Thm. 2.7]{AIR2014} we have bijection between $\tau$-tilting pairs and $\ftors A$, which induced poset structure on the set of $\tau$-tilting pairs. By \cite{IRTT15} $\ftors A$ forms a complete lattice if and only if $\ftors A = \tors A$ if and only if there are only finitely many $\tau$-rigid pairs. An algebra $A$ satisfying these equivalent conditions is called \textit{$\tau$-tilting finite}.

\begin{prop}\cite[Cor. 2.13]{AIR2014}
    Every $\tau$-rigid pair $(M,P)$ gives rise to two torsion classes $\Fac M \in \ftors A$ and ${}^\perp \tau M \cap P^\perp \in \ftors A$. Moreover, $\Fac M \subseteq {}^\perp \tau M \cap P^\perp$.
\end{prop}

More generally, given a full subcategory $\Scal \subseteq \mods A$, we denote by  $\Trm(\Wcal) \coloneqq \Filt (\Fac \Scal)$ the smallest torsion class containing $\Scal$, see \cite[Lem. 3.1]{MarksStovicek}.

\begin{defn}
    An interval $[\Ucal, \Tcal]\subseteq \tors A$ is called \textit{$\tau$-perpendicular} if $\Ucal = \Fac M$ and $\Tcal = {}^\perp \tau M \cap P^\perp$ for some $\tau$-rigid pair $(M,P)$. A \textit{$\tau$-perpendicular wide subcategory} $W$ of $\mods A$ is a wide subcategory $W= M^\perp \cap {}^\perp \tau M \cap P^\perp$ for some $\tau$-rigid module $(M,P)$.
\end{defn}

Such $\tau$-perpendicular intervals of $\tors A$ have also been called polytopes \cite{DIRRT2017}, wide intervals \cite{AsaiPfeifer2019} or binuclear intervals \cite{Hanson2024} since they can be described combinatorially. More importantly, they arise naturally in the process of $\tau$-tilting reduction \cite[Thm. 3.12]{Jasso2015}, see also \cite[Thm. 4.12]{DIRRT2017}.

\begin{thm}
    Let $(M,P)$ be a $\tau$-rigid pair then $[\Fac M, {}^\perp \tau M \cap P^\perp] \subseteq \tors A$ is isomorphic to $\tors \Wcal_{(M,P)}$ where $\Wcal_{(M,P)} \coloneqq M^\perp \cap {}^\perp \tau M \cap P^\perp$ and there is an equivalence of categories $\Wcal_{(M,P)} \cong \mods C$ with some finite-dimensional algebra $C_{(M,P)}$ the $\tau$-tilting reduction of $A$ with respect to $(M,P)$.
\end{thm}

Since $\tau$-perpendicular categories $\Wcal$ are equivalent to module categories, they have their own (relative) Auslander-Reiten translation $\tau_{\Wcal}$ and (relative) $\tau$-rigid pairs. We are now able to give the algebraic definition of the $\tau$-cluster morphism category in the general setting of finite-dimensional algebras. For definitions in hereditary and $\tau$-tilting finite settings see \cite{IgusaTodorov2017} and \cite{BuanMarsh2018}, respectively. For definitions in more general settings see \cite{Borve2022} and \cite{Kaipelcatpartfan}.

\begin{defn}\cite[Def. 6.1]{BuanHanson2017}
    The $\tau$-cluster morphism category $\Wfrak(A)$ has as its objects the $\tau$-perpendicular wide subcategories of $\mods A$. Given two $\tau$-perpendicular wide subcategories $\Wcal_1, \Wcal_2$ of $\mods A$, define
    \[ \Hom_{\Wfrak(A)} ( \Wcal_1, \Wcal_2) = \{g_{(M,P)}^{\Wcal_1} : (M,P) \text{ is a basic $\tau_{\Wcal_1}$-rigid pair and $\Wcal_2 = M^{\perp_{\Wcal_1}} \cap {}^{\perp_{\Wcal_1}} \tau_{\Wcal_1} M \cap P^{\perp_{\Wcal_1}}$}\}.\]
\end{defn}

We choose to omit the details on composition of morphisms for this category, whose associativity is highly non-trivial to prove, see \cite[Thm. 5.9]{BuanMarsh2018} and \cite[Thm. 6.12]{BuanHanson2017}. The generalisation of $\Wfrak(A)$ to non-positive dg algebras with finite-dimensional cohomology in all degrees of \cite{Borve2022} specialises to the above for finite-dimensional algebras and replaces $\tau$-perpendicular categories by certain thick subcategories of the bounded derived category $\Dcal^b(\mods A)$. The $\tau$-tilting reduction then becomes silting reduction \cite{IyamaYang2018}, whose functoriality implies the associativity more directly \cite[Thm. 4.3]{Borve2022}.

\subsection{Geometric definition}
Following \cite{DIJ2019}, we associate a polyhedral fan $\Sigma(A)$ to every algebra $A$ whose cones encode the $\tau$-tilting theory \cite{Asai2019WS,BST2019}. Let $M \in \mods A$ an consider a minimal projective presentation $\bigoplus_{i=1}^{|A|} P(i)^{b_i} \to \bigoplus_{i=1}^{|A|} P(i)^{a_i} \to M \to 0$. The \textit{$g$-vector} $g^M$ is defined as
\[ g^M \coloneqq (a_1 - b_1, \dots, a_{|A|} - b_{|A|})^T \in \R^{|A|}.\]

\begin{defn}
    The \textit{$g$-vector fan} $\Sigma(A)$ is given as the union of positive linear combinations of $g$-vectors, called \text{$g$-vector cones} of $\tau$-rigid pairs as follows:
    \[ \Sigma(A) \coloneqq \{ \Ccal_{(M,P)}: (M,P) \text{ is a $\tau$-rigid pair}\},\]
    where $\Ccal_{(M,P)} \coloneqq \spann_{\geq 0} \{ g^{M^1}, \dots, g^{M^r}, - g^{P_{r+1}}, \dots, g^{P_s}\} \subseteq \R^{|A|}$ and $M^i$ and $P^i$ are indecomposable direct summands of $M$ and $P$ such that $(M,P) = (\bigoplus_{i=1}^r M^i, \bigoplus_{i={r+1}}^s P^i)$.
\end{defn}

\begin{rmk}
    When $\Sigma(A)$ is finite, the poset $\tors A$ induces a fan poset, see \cite{Reading2003}, on the maximal cones of $\Sigma(A)$ by \cite[Prop. 6.15]{Kaipelcatpartfan}.
\end{rmk}

For a $g$-vector cone $\Ccal_{(M,P)} \in \Sigma(A)$ denote by $\pi_{\Ccal_{(M,P)}}: \R^n \to \spann \{ \Ccal_{(M,P)}\}^\perp$ the orthogonal projection onto the subspace orthogonal to the linear span of $\Ccal_{(M,P)}$. Intuitively, this map corresponds to $\tau$-tilting reduction present in the previous definition, see \cite[Sec. 4]{Asai2019WS}. The authors of \cite{STTW2023} use the information encoded in the $g$-vector fan to define the following category.

\begin{defn}\cite[Def. 3.3]{STTW2023}\label{defn:STTWdefn}
    The \textit{(geometric) $\tau$-cluster morphism category} $\Cfrak(A)$ has as its objects equivalence classes of $g$-vector cones $[\Ccal_{(M,P)}]$ under the equivalence relations $\Ccal_{(M_1,P_1)} \sim \Ccal_{(M_2,P_2)}$ whenever $\Wcal_{(M_1,P_1)} = \Wcal_{(M_2,P_2)}$. The morphisms of $\Cfrak(A)$ are given by equivalence classes of morphisms in the poset category of the fan $\Sigma(A)$ under inclusion of $g$-vector cones. More precisely,
    \[ \Hom_{\Cfrak(A)}([\Ccal_{(M,P)}],[\Ccal_{(N,Q)}]) = \bigcup_{\substack{\Ccal_{(M',P')} \in [ \Ccal_{(M,P)}]\\ \Ccal_{(N',Q')} \in [\Ccal_{(N,Q)}]}} \Hom_{\Sigma(A)}(\Ccal_{(M',P')}, \Ccal_{(N',Q')})\]
    under the equivalence relation $f_{\Ccal_{(M_1,P_1)}\Ccal_{(N_1,Q_1)}} \sim f_{\Ccal_{(M_2,P_2)}\Ccal_{(N_2,P_2)}}$ whenever 
    \[\pi_{\Ccal_{(M_1,P_1)}}(\Ccal_{(N_1,Q_1)}) = \pi_{\Ccal_{(M_2,P_2)}}(\Ccal_{(N_2,Q_2)}). \]
\end{defn}

This approach makes the associativity of the composition of morphisms trivial, but some technicalities remain to show composition is well-defined, see \cite[Lem. 3.9, 3.10]{STTW2023}.

\begin{thm} \cite[Thm. 3.11, Cor 4.6]{STTW2023}
    $\Cfrak(A)$ is a well-defined category and $\Cfrak(A) \cong \Wfrak(A)$ is an equivalence of categories.
\end{thm}

The work of \cite{STTW2023} inspired the author to generalise the construction above to simplicial polyhedral fans in \cite{Kaipelcatpartfan}, where it was shown that different choices of identifications of cones also give well-defined categories. More importantly, this generalisation provided new insight into properties of $\Cfrak(A)$.

\section{Definition from the lattice of torsion classes}\label{sec:definition}
We now take an alternative approach to \cite{STTW2023} and define the $\tau$-cluster morphism category from the lattice of torsion classes. Let $\itv(\tors A)$ be the poset of intervals of $\tors A$ ordered (conventionally) by reverse containment. If $[\Ucal, \Tcal], [\Scal, \Vcal] \in \itv( \tors A)$ are two intervals, we say $[\Ucal, \Tcal] \leq [\Scal, \Vcal]$ whenever $\Ucal \subseteq \Scal$ and $\Vcal \subseteq \Tcal$. The relationship of this construction with the geometric one comes from the assignment 
\begin{equation}\label{eq:gfantoitv}
\begin{aligned}
    \Sigma(A) &\to \itv(\tors A) \\
    \Ccal_{(M,P)} &\mapsto [\Fac M, {}^\perp \tau M \cap P^\perp].
\end{aligned}
\end{equation}

In particular, considering only maximal cones $\Ccal_{(M,P)} \in \Sigma(A)^n$ in \cref{eq:gfantoitv}, where $(M,P)$ is $\tau$-tilting and ${}^\perp \tau M \cap P^\perp = \Fac M$ by \cite[Prop. 2.16b]{AIR2014}, we observe the bijection between $\tau$-tiltling modules $(M,P)$ and functorially-finite torsion classes $\Fac M \in \ftors A$ of \cite[Thm. 2.7]{AIR2014}.  \\ 

Recall, that the \textit{Hasse quiver} $\Hasse(P)$ of a poset $P$ is the quiver whose vertices are elements $x \in P$ and whose arrows $y \to x$ correspond with cover relations $x \lessdot y$ in $P$. Recall that we write $x \lessdot y$ if there does not exist $z \in P$ such that $x \lneq z \lneq y$ holds. Furthermore, for a subcategory $\Scal \subseteq \mods A$ we write $\Filt_A(\Scal)$ for the full subcategory of $\mods A$ filtered by modules in $\Scal$ and for a module $M \in \mods A$ we write $\Filt_A(M) \coloneqq \Filt_A(\add\{M\})$. \\

An important part of $\tors A$ is its \textit{brick-labelling} \cite{AsaiSB2020, DIRRT2017}. A \textit{brick} $B \in \mods A$ is module whose endomorphism algebra $\End_A(B)$ is a division algebra. Let $\Ucal \lessdot \Tcal$ be a covering relation in $\tors A$, then we can label the arrow $\Tcal \xrightarrow{B} \Ucal$ in $\Hasse(\tors A)$ with the unique brick $B$ such that $\Ucal^\perp \cap \Tcal = \Filt_A (B)$, by \cite[Thm. 3.3]{DIRRT2017} or \cite[Lem. 1.16]{AsaiSB2020}. For an interval $[\Ucal, \Tcal] \subseteq \tors A$ we write $\brick[\Ucal, \Tcal]$ for the set of bricks arising as brick labels of the interval, or equivalent, the bricks in $\Ucal^\perp \cap \Tcal$. \\

The analogue of the map $\pi_{\Ccal_{(M,P)}}$ in the geometric definition, which allows us to identify morphisms corresponding to ``the same'' $\tau$-tilting reduction, is the following result. For a $\tau$-rigid pair $(M,P)$ denote by $[\Ucal_{(M,P)}, \Tcal_{(M,P)}]$ the interval $[\Fac M, {}^\perp \tau M \cap P^\perp]$.

\begin{thm}\cite[Thm. 4.12]{DIRRT2017}\label{thm:wideintiso}
    Let $[\Ucal_{(M_1,P_1)}, \Tcal_{(M_1,P_1)}]$ and $[\Ucal_{(M_2,P_2)}, \Tcal_{(M_2,P_2)}]$ be two distinct $\tau$-perpendicular intervals of $\tors A$ such that $\Wcal_{(M_1,P_1)} = \Wcal_{(M_2,P_2)}$. Then there are three lattice isomorphisms
    \[
    \begin{tikzcd}
        {[}\Ucal_{(M_1,P_1)}, \Tcal_{(M_1,P_1)}{]} \arrow[rr, "\cong"] \arrow[rd, "{- \cap \Wcal_{(M_1,P_1)}}", swap] && {[}\Ucal_{(M_2,P_2)}, \Tcal_{(M_2,P_2)}{]} \arrow[ld, "{- \cap \Wcal_{(M_2,P_2)}}"]\\
        & \tors \Wcal_{(M_1,P_1)} = \tors \Wcal_{(M_2,P_2)} 
    \end{tikzcd}
    \]
\end{thm}

The isomorphism in between the $\tau$-perpendicular intervals factors through the lattice of torsion classes of their shared $\tau$-perpendicular wide subcategory. Moreover, the following is an important observation.

\begin{prop}\cite[Prop. 4.13]{DIRRT2017}\label{prop:brickpreserve}
    The three lattice isomorphisms in \cref{thm:wideintiso} preserve the brick-labelling of Hasse quivers.
\end{prop}

We now present the definition of the $\tau$-cluster morphism category from the lattice of torsion classes.

\begin{defn}\label{defn:latticedef}
    The \textit{(lattice-theoretic) $\tau$-cluster morphism category} $\Tfrak(A)$ has as its objects equivalence classes $[\Ucal_{(M,P)}, \Tcal_{(M,P)}]_\sim$ of $\tau$-perpendicular intervals of $\tors A$ under the equivalence relation 
    \[[\Ucal_{(M_1,P_1)}, \Tcal_{(M_1,P_1)}] \sim [\Ucal_{(M_2,P_2)}, \Tcal_{(M_2,P_2)}] \]
    whenever $\Wcal_{(M_1,P_1)} = \Wcal_{(M_2,P_2)}$. The morphisms of $\Tfrak(A)$ are given by equivalence classes of morphisms in the poset category $\itv (\tors A)$. More precisely,
    \begin{align*}
            &\Hom_{\Tfrak(A)}([\Ucal_{(M,P)}, \Tcal_{(M,P)}]_\sim, [\Ucal_{(N,Q)}, \Tcal_{(N,Q)}]_\sim) \\
            &=\bigcup_{\substack{[\Ucal_{(M',P')}, \Tcal_{(M',P')}] \in [\Ucal_{(M,P)}, \Tcal_{(M,P)}]_\sim \\ [\Ucal_{(N',Q')}, \Tcal_{(N',Q')}] \in [\Ucal_{(N,Q)}, \Tcal_{(N,Q)}]_\sim}} \Hom_{\itv(\tors A)} ( [\Ucal_{(M',P')}, \Tcal_{(M',P')}], [\Ucal_{(N',Q')}, \Tcal_{(N',Q')}])
        \end{align*}
    under the equivalence relation
    \[f_{[\Ucal_{(M_1,P_1)}, \Tcal_{(M_1,P_1)}][\Ucal_{(N_1,Q_1)}, \Tcal_{(N_1,Q_1)}]} \sim f_{[\Ucal_{(M_2,P_2)}, \Tcal_{(M_2,P_2)}][\Ucal_{(N_2,Q_2)}, \Tcal_{(N_2,Q_2)}]}\]
    whenever $[\Ucal_{(N_1, Q_1)}, \Tcal_{(N_1,Q_1)}] \cap \Wcal_{(M_1, P_1)} = [\Ucal_{(N_2, Q_2)}, \Tcal_{(N_2,Q_2)}] \cap \Wcal_{(M_2, P_2)}$.
\end{defn}

It is not difficult to see that there are bijections between the objects of $\Cfrak(A)$ and those of $\Tfrak(A)$ via the correspondence \cref{eq:gfantoitv}.  When the reference to the specific $\tau$-rigid pair is not important in later sections, we denote $\tau$-perpendicular intervals by $[\Ucal, \Tcal]$ and suppress writing $(M,P)$ which shortens the notation.\\

Composition of morphisms in $\Tfrak(A)$ is trivially associative. However, similar ambiguities to those in the geometric construction of \cite{STTW2023} arise when determining whether composition in $\Tfrak(A)$ is well-defined, because we are also defining $\Tfrak(A)$ via equivalence classes. We begin by connecting the objects and morphisms of $\Cfrak(A)$ and $\Tfrak(A)$ showing that identifications happen in the same way.

\begin{lem} \label{lem:objtranslation}
    The assignment of \cref{eq:gfantoitv} induces a well-defined bijection between the objects of $\Cfrak(A)$ and $\Tfrak(A)$.
\end{lem}
\begin{proof}
    Using only the definitions of $\Cfrak(A)$ and $\Tfrak(A)$, we have the following sequence of implications
    \begin{align*}
        \Ccal_{(M_1,P_1)} \sim_{\Cfrak(A)} \Ccal_{(M_2,P_2)} \Leftrightarrow  \Wcal_{(M_1,P_1)} = \Wcal_{(M_2,P_2)} \Leftrightarrow [\Ucal_{(M_1,P_1)}, \Tcal_{(M_1,P_1)}] \sim_{\Tfrak(A)} [\Ucal_{(M_2,P_2)}, \Tcal_{(M_2,P_2)}].
    \end{align*}
    Therefore the equivalence relation is preserved on objects.
\end{proof}

\begin{lem}\label{lem:morphtranslation}
    The assignment of \cref{eq:gfantoitv} induces a well-defined map between morphisms of $\Cfrak(A)$ and $\Tfrak(A)$. That is, the identification of morphisms is preserved, or in other words, for two inclusions of $g$-vector cones $\Ccal_{(M_i,P_i)} \subseteq \Ccal_{(N_i, Q_i)} \in \Sigma(A)$ for $i=1,2$ with corresponding $\tau$-perpendicular intervals $[\Ucal_{(M_i,P_i)}, \Tcal_{(M_i,P_i)}] \supseteq [\Ucal_{(N_i,Q_i)}, \Tcal_{(N_i,Q_i)}]$ we have  
    \begin{align*}
         &\pi_{\Ccal_{(M_1,P_1)}}(\Ccal_{(N_1,Q_1)}) = \pi_{\Ccal_{(M_2,P_2)}}(\Ccal_{(N_2,Q_2)}) \\
         &\Leftrightarrow ([\Ucal_{(N_1,Q_1)}, \Tcal_{(N_1,Q_1)}]) \cap \Wcal_{(M_1,P_1)} = [\Ucal_{(N_2,Q_2)}, \Tcal_{(N_2,Q_2)}] \cap \Wcal_{(M_2,P_2)}.
    \end{align*}
\end{lem}
\begin{proof}
    This follows from \cite[Lem. 3.8]{STTW2023} and \cite[Lem. 4.4(3)]{Asai2019WS}. More precisely, \cite[Lem. 3.8]{STTW2023} is used to translate between the orthogonal projection map $\pi_{\overline{\Ccal}_{(M,P)}}: \R^{|A|} \to \spann\{ \Ccal_{(M,P)}\}^\perp$ as defined in \cref{defn:STTWdefn} and the map $\pi: K_0(\proj A) \to K_0(\proj \Wcal_{(M,P)})$ defined in \cite[p. 33]{Asai2019WS}. Then \cite[Lem. 4.4(3)]{Asai2019WS} implies that 
   \begin{align*}
   	&\pi_{\overline{\Ccal}_{(M_1,P_1)}}(\overline{\Ccal}_{(N_1,Q_1)}) = \pi_{\overline{\Ccal}_{(M_2,P_2)}}(\Ccal_{(N_2,Q_2)}) \\
	& \Leftrightarrow \Tcal_{(N_1,Q_1)} \cap \Wcal_{(M_1,P_1)} = \Tcal_{(N_1, Q_1)} \cap \Wcal_{(M_2,P_2)} \text{ and }  \Ucal_{(N_1,Q_1)} \cap \Wcal_{(M_1,P_1)} = \Ucal_{(N_1, Q_1)} \cap \Wcal_{(M_2,P_2)} \\
	& \Leftrightarrow ([\Ucal_{(N_1,Q_1)}, \Tcal_{(N_1,Q_1)}]) \cap \Wcal_{(M_1,P_1)} = [\Ucal_{(N_2,Q_2)}, \Tcal_{(N_2,Q_2)}] \cap \Wcal_{(M_2,P_2)}.
   \end{align*}
   Whence the result follows.
\end{proof}

Now we show that composition in $\Tfrak(A)$ is well-defined. The first concern, see \cite[Lem. 3.9]{STTW2023} is that it is not clear how to compose two morphisms $[f_{[\Ucal, \Tcal][\Vcal_1, \Scal_1]}]$ and $[f_{[\Vcal_2, \Scal_2][\Xcal_2, \Ycal_2]}]$ when $[\Vcal_1, \Scal_1] \sim [\Vcal_2, \Scal_2]$. The following lemma implies the existence of a morphism $f_{[\Vcal_1, \Scal_1][\Xcal_1, \Ycal_1]} \sim f_{[\Vcal_2, \Scal_2][\Xcal_2, \Ycal_2]}$ so that composition of the two morphisms is simply $[f_{[\Ucal, \Tcal][\Xcal_1, \Ycal_1]}]$.

\begin{lem} \label{lem:composition1}
    Let $(M_1,P_1),(M_2,P_2)$ be $\tau$-rigid pairs such that $\Wcal(M_1,P_1) = \Wcal(M_2,P_2)$. Then for every $\tau$-perpendicular interval 
    \[[\Ucal_{(N_1,Q_1)}, \Tcal_{(N_1,Q_1)}] \subseteq [\Ucal_{(M_1,P_1)}, \Tcal_{(M_1,P_1)}]\]
    there exists a $\tau$-perpendicular interval
    \[ [\Ucal_{(N_2,Q_2)}, \Tcal_{(N_2,Q_2)}] \subseteq [\Ucal_{(M_2,P_2)}, \Tcal_{(M_2,P_2)}]\]
    such that $\Wcal_{(N_1, Q_1)} = \Wcal_{(N_2,Q_2)}$.
\end{lem}
\begin{proof}
    There exists an explicit bijection $E_{(M_i,P_i)}$ from $\tau$-rigid pairs in $\mods A$ containing $(M_i,P_i)$ as a direct summand to $\tau$-rigid pairs in $\Wcal(M_i,P_i)$ for $i=1,2$, see \cite[Sec. 5]{BuanMarsh2021} and \cite[Def. 3.5]{BKT2025}. We define $(N_2,Q_2) = E_{(M_2,P_2)}^{-1} (E_{(M_1,P_1)}(N_1,Q_1))$. Moreover, these map are order-preserving on $\tau$-tilting pairs by \cite[Thm. 3.15]{Jasso2015}, see also \cite[Thm. 3.8]{BKT2025}. It follows that $\phi(\Ucal_{(N_1,Q_1)}) = \Ucal_{(N_2,Q_2)}$ and $\phi(\Tcal_{(N_1,Q_1)})= \Ucal_{(N_2,Q_2)}$. Hence 
    \[ \phi([\Ucal_{(N_1,Q_1)}, \Tcal_{(N_1,Q_1)}]) = [\Ucal_{(N_2,Q_2)}, \Tcal_{(N_2,Q_2)}].\]
    Moreover, $\phi$ preserves the brick-labelling by \cref{prop:brickpreserve}, thus \cite[Thm. 5.2, Prop. 5.3]{AsaiPfeifer2019}, see also \cite[Thm. 4.16]{DIRRT2017}, implies that $\Wcal_{(N_2,Q_2)} = \Wcal_{(N_1, Q_1)}$ as required.
\end{proof}

The second concern is whether the composition of pairwise-identified morphisms gives identified morphisms, see \cite[Lem. 3.10]{STTW2023}. However, this is easily seen to hold via the lattice isomorphism of \cref{thm:wideintiso}. We have thus constructed the $\tau$-cluster morphism category from $\tors A$.

\begin{thm}\label{thm:welldefn}
    The category $\Tfrak(A)$ is well-defined and equivalent to $\Cfrak(A)$.
\end{thm}
\begin{proof}
    The category is well-defined by \cref{lem:composition1}. The equivalence $\Psi: \Cfrak(A) \to \Tfrak(A)$ is induced by \cref{eq:gfantoitv}. We have seen in \cref{lem:objtranslation} and \cref{lem:morphtranslation} that $\Psi$ is well-defined on objects and morphisms. By definition, every $\tau$-perpendicular interval comes from a $\tau$-rigid pair $(M,P)$ and hence from a $g$-vector cone $\Ccal_{(M,P)}$ via \cref{eq:gfantoitv}, so $\Psi$ is essentially surjective. For the same reason, given two $\tau$-rigid pairs $(M,P)$ and $(N,Q)$, we have $\Ccal_{(N,Q)} \subseteq \Ccal_{(M,P)}$ if and only if $[\Ucal_{(N,Q)}, \Tcal_{(N,Q)}] \leq [\Ucal_{(M,P)},\Tcal_{(M,P)}]$. This determines the morphisms of $\Cfrak(A)$ and $\Tfrak(A)$ respectively, therefore we have bijections between $\Hom$-sets. So $\Psi$ is fully faithful and hence an equivalence.
\end{proof}

\begin{rmk}\label{rmk:latticecapturing}
    The structure of the poset of functorially-finite torsion classes and its brick labelling is encoded in the morphism spaces of the $\tau$-cluster morphism category. For example, there is a bijection
    \begin{align*} \Hom_{\Tfrak(A)}([0, \mods A]_\sim, [0,0]_\sim) \quad \longleftrightarrow \quad \Tcal \in \ftors A\end{align*}
    since $[0,0]_\sim = [\Tcal, \Tcal]_\sim$ for all $\Tcal \in \ftors A$. Each such interval is $\tau$-perpendicular because functorially finite torsion classes correspond bijectively to $\tau$-tilting pairs $(M,P)$ for which $\Fac M = {}^\perp \tau M \cap P^\perp$ by \cite[Prop. 2.16]{AIR2014}. Moreover, let $\Tcal_2 \xleftarrow{B} \Tcal_1 \subseteq \Hasse(\ftors A)$ be a cover relation labelled by the brick $B$, then there is a bijection
    \begin{align*}
    \Hom_{\Tfrak(A)}([0, \mods A]_\sim, [\Tcal_2, \Tcal_1]_\sim) \quad \longleftrightarrow \quad \{ \text{arrows labelled by $B$ in $\Hasse(\ftors A)$}\}. \end{align*}
    And more generally for an arbitrary $\tau$-perpendicular interval $[\Tcal_4, \Tcal_3] \subseteq \tors A$
    \begin{align*} &\Hom_{\Tfrak(A)}([0, \mods A]_\sim, [\Tcal_4, \Tcal_3]_\sim) \\ 
    &\longleftrightarrow \{ \text{$\tau$-perpendicular intervals of $\tors A$ with a brick-label preserving lattice isomorphism to $[\Tcal_4, \Tcal_3]$}\}. \end{align*}
    The same holds true when viewing any $\tau$-perpendicular interval $[\Ucal, \Tcal] \subseteq \tors A$ as the domain of the morphism space and restricting to $\tors [\Ucal, \Tcal]$. 
\end{rmk}

\section{Invariance under $\tau$-tilting equivalence} \label{sec:tauequiv}
In this section we consider the notion of $\tau$-tilting equivalence of two algebras as the existence of an isomphism between their respective posets of $\tau$-tilting pairs. The idea of using only lattice-theoretic information to study $\tau$-tilting theory is common, see for example \cite{AsaiPfeifer2019, BarnardCarrolZhu19, BarnardHanson2022exc, DIRRT2017, Enomoto2023, kase2017}. Using ideas similar to these authors, we show how \cref{defn:latticedef} can be rephrased so that the $\tau$-cluster morphism category can be determined only from the underlying poset of (functorially-finite) torsion classes when it is finite.

\begin{defn}
Let $A$ and $B$ be two finite-dimensional algebras. We say that $A$ and $B$ are \textit{$\tau$-tilting equivalent} if there exists a poset isomorphism $\ftors A \cong \ftors B$. 
\end{defn}

\begin{exmp}\label{exmp:tauequiv}The following algebras are $\tau$-tilting equivalent:
\begin{enumerate}
\item Any algebra $A$ and $A/ \langle c \rangle$, where $c \in A$ is a central element contained in the Jacobson radical \cite{EJR2018}. See also \cite[Cor. 5.20]{DIRRT2017}.
\item Two algebras $A$ and $B$ which geometrically have coinciding $g$-vector fan, by the duality of the poset $\ftors A$ with the chambers of the $g$-vector fan, by \cite{DIJ2019}.
\item The algebras $K_2 \coloneqq K(1 \xrightarrow{2} 2)$ and $K_3 \coloneqq K(1 \xrightarrow{m} 2)$, for any $m \geq 3$, by the description of $\ftors (K_2)$ and $\ftors (K_3)$ as polygons with one infinite side. 
\end{enumerate}
\end{exmp}

We note that $\ftors A$ is a lattice if and only if $A$ is \textit{$\tau$-tilting finite} i.e. if $\tors A = \ftors A$ \cite{IRTT15}. However, since $\ftors A \subseteq \tors A$, we can consider the join and meet of elements of $\ftors A$ in $\tors A$. The first example of $\tau$-tilting theory of $\mods A$ being encoded lattice-theoretically in $\ftors A$ is that of $\tau$-perpendicular intervals. 

\begin{prop}\cite[Prop. 4.19 and 4.20]{DIRRT2017}\label{prop:DIRRT420}
    Let $\Ucal \in \ftors A$. Consider $\ell$ atoms $\Ucal_i \to \Ucal \subseteq \Hasse(\ftors A)$. Then $\Tcal \coloneqq \bigvee_{i=1}^\ell \Ucal_i \in \ftors A$ and $[\Ucal, \Tcal] \subseteq \tors A$ is a $\tau$-perpendicular interval of $\tors A$. Moreover, every $\tau$-perpendicular interval is a so-called \emph{join-interval} of this form.
\end{prop}

In \cite{AsaiPfeifer2019} it is shown the construction of \cref{prop:DIRRT420} gives all intervals $[\Ucal, \Tcal] \subseteq \tors A$ for which $\Ucal^\perp \cap \Tcal$ is a wide subcategory of $\mods A$. To obtain only the $\tau$-perpendicular subcategories, we require knowledge of the subset $\ftors A \subseteq \tors A$ which cannot be determined lattice-theoretically from $\tors A$. For this reason we restrict ourselves to the case when $A$ is $\tau$-tilting finite, which implies $\ftors A = \tors A$ by \cite[Thm. 3.8]{DIJ2019}. \\

In this case, $\tau$-perpendicular intervals and join-intervals (and meet-intervals) all coincide. In order to distinguish between intervals of the same shape but which correspond to distinct subcategories we need to recover the brick labelling of $\tors A$ combinatorially, since the brick labelling determines the corresponding $\tau$-perpendicular subcategory by \cite[Thm. 4.16]{DIRRT2017} and \cite[Prop. 5.3]{AsaiPfeifer2019}.\\

Let $L$ be a lattice, an element $j \in L$ is called \textit{completely join-irreducible} if there does not exist a subset $S \subseteq L$ such that $j = \bigvee S$ and $j \not \in S$. A join-irreducible element $j \in L$ covers exactly one other element $j_*$. Take $x \lessdot y$ then the set $\{ t \in L : t \lor x = y\} \subseteq L$ has a minimum element $\ell \in L$ which is completely join-irreducible and satisfies $\ell_* \leq x$ by \cite[Lem. 3.7]{FTFSL}. This inspires the following definition, similar to \cite[Rmk. 3.8]{FTFSL}.

\begin{defn}
Let $L$ be a completely semidistributive lattice. The \textit{join-irreducible labelling} of $L$ associates to each arrow $x \leftarrow y \subseteq \Hasse(L)$ the unique completely join-irreducible element $j$ which is minimal in the set $\{ t \in L : x \lor t = y\} \subseteq L$.
\end{defn}

\begin{lem} \cite[Thm. 3.11]{DIRRT2017} \label{lem:labelcoincide}
Let $L$ be an abstract lattice isomorphic to $\tors A$, then the join-irreducible labelling on $L$ corresponds to the brick labelling of $\tors A$. In other words, the brick labelling is determined combinatorially by the underlying lattice structure.
\end{lem}

\begin{rmk}
    This fact was used in \cite{Enomoto2023} to recover the posets of wide and so-called ICE-closed subcategories of $\mods A$ from the lattice theoretic information of $\tors A$. And we have taken inspiration for the formulation of the following theorem from \cite[Thm. A]{Enomoto2023}.
\end{rmk}

For an interval $[x,y] \subseteq L$ write $\jirrc [x,y]$ for the set of completely-join irreducible elements $j \in L$ arising as labels of cover relations of $x \leq a \lessdot b \leq y$. The following is the main result of this section and follows the same idea as the previous remark. 

\begin{thm}\label{thm:categoryiso}
Let $L$ be an abstract finite lattice isomorphic to $\tors A$ for some finite-dimensional algebra $A$. Then we can define $\Tfrak(A)$ combinatorially from $L$ without using any algebraic information of $A$ or $\mods A$. More precisely, we have an equivalence of categories $\Tfrak(L) \cong \Tfrak(A)$, where $\Tfrak(L)$ is defined to be the category:
\begin{itemize}
	\item whose objects are equivalence classes of join-intervals $[\Ucal, \Tcal] = [\Ucal, \bigvee_{i=1}^\ell \Ucal_i] \subseteq L$ under the equivalence
	\[ [\Ucal_1, \Tcal_1] \sim [\Ucal_2, \Tcal_2] \]
	whenever $\jirrc[\Ucal_1, \Tcal_1] = \jirrc[\Ucal_2, \Tcal_2]$. In this case there exists an isomorphism  $\phi_{[\Ucal_1,\Tcal_1][\Ucal_2, \Tcal_2]}: [\Ucal_1, \Tcal_1] \to [\Ucal_2, \Tcal_2]$ preserving the join-irreducible labels (in $L$) of the arrows in the Hasse diagrams;
	\item and whose morphisms are given by equivalence classes of morphisms in the poset category $\itv (L)$ of intervals of $L$ partially-ordered by reverse containment. More precisely
	\begin{align*}
            \Hom_{\Tfrak(L)}([\Ucal, \Tcal]_\sim, [\Vcal, \Scal]_\sim) =\bigcup_{\substack{[\Ucal', \Tcal'] \in [\Ucal, \Tcal]_\sim \\ [\Vcal', \Scal'] \in [\Vcal, \Scal]_\sim}} \Hom_{\itv(L)} ( [\Ucal', \Tcal'], [\Vcal', \Scal'])
        \end{align*}
    under the equivalence relation $f_{[\Ucal_1, \Tcal_1][\Vcal_1, \Scal_1]} \sim f_{[\Ucal_2, \Tcal_2][\Vcal_2, \Scal_2]}$
    whenever 
    \[\phi_{[\Ucal_1, \Tcal_1][\Ucal_2, \Tcal_2]}( [\Vcal_1, \Scal_1]) = [\Vcal_2, \Scal_2]. \]
\end{itemize}
\end{thm}
\begin{proof}
Denote by $\psi: L \to \tors A$ the isomorphism of complete lattices. Since $L$ is finite it follows that $\tors A$ is finite and hence $\ftors A = \tors A$ by \cite[Thm. 3.8]{DIJ2019}. Let $[\Ucal_1, \Tcal_1]$ and $[\Ucal_2, \Tcal_2]$ be two distinct join-intervals of $L$ satisfying $\jirrc[\Ucal_1, \Tcal_1] = \jirrc[\Ucal_2, \Tcal_2]$. By \cref{prop:DIRRT420} and \cref{lem:labelcoincide}, this is equivalent to saying that $\psi([\Ucal_1, \Tcal_1])$ and $\psi([\Ucal_2, \Tcal_2])$ are $\tau$-perpendicular intervals of $\tors A$ such that $\brick(\psi([\Ucal_1, \Tcal_1])) = \brick(\psi([\Ucal_2, \Tcal_2]))$. Moreover, by \cite[Lem. 3.10]{DIRRT2017} we have
\[ \psi(\Ucal_1)^\perp \cap \psi(\Tcal_1) = \Filt(\brick(\psi([\Ucal_1,\Tcal_1]))) = \Filt(\brick(\psi([\Ucal_2, \Tcal_2]))) = \psi(\Ucal_2)^\perp \cap \psi(\Tcal_2).\]

It follows that $[\Ucal_1, \Tcal_1] \sim_{\Tfrak(L)} [\Ucal_2, \Tcal_2]$ implies $\psi([\Ucal_1, \Tcal_1]) \sim_{\Tfrak(A)} \psi([\Ucal_2, \Tcal_2])$. Moreover, then \cref{thm:wideintiso} and \cref{prop:brickpreserve} imply that there is a brick label-preserving isomorphism $\varphi: \psi([\Ucal_1, \Tcal_1]) \to \psi([\Ucal_2, \Tcal_2])$. By \cref{lem:labelcoincide} $\varphi$ lifts to a join-irreducible-label-preserving isomorphisms 
\[ \phi_{[\Ucal_1, \Tcal_1][\Ucal_2, \Tcal_2]} \coloneqq \psi^{-1} \varphi \psi: [\Ucal_1, \Tcal_1] \to [\Ucal_2 , \Tcal_2]. \]

Let $[\Vcal_i, \Scal_i] \subseteq L$ be join-intervals such that $[\Vcal_i, \Scal_i]  \subseteq [\Ucal_i, \Tcal_i]$ for $i=1,2$. Then $f_{[\Ucal_1, \Tcal_1][\Vcal_1, \Scal_1]} \sim_{\Tfrak(L)} f_{[\Ucal_2, \Tcal_2][\Vcal_2, \Scal_2]}$ means $\phi_{[\Ucal_1, \Tcal_1][\Ucal_2, \Tcal_2]}([\Vcal_1, \Scal_1]) = [\Vcal_2, \Scal_2]$ which implies 
\[ \psi([\Vcal_1, \Scal_1]) \cap (\psi(\Ucal_1)^\perp \cap \psi(\Tcal_1)) = \psi([\Vcal_2, \Scal_2]) \cap (\psi(\Ucal_2)^\perp \cap \psi(\Tcal_2))\]
as $\varphi$ factors through $\tors (\psi(\Ucal_1)^\perp \cap \psi(\Tcal_1)) = \tors (\psi(\Ucal_2)^\perp \cap \psi(\Tcal_2))$. Thus 
\[ f_{[\Ucal_1, \Tcal_1][\Vcal_1, \Scal_1]} \sim_{\Tfrak(L)} f_{[\Ucal_2, \Tcal_2][\Vcal_2, \Scal_2]} \Longrightarrow f_{\psi([\Ucal_1, \Tcal_1])\psi([\Vcal_1, \Scal_1])} \sim_{\Tfrak(A)} f_{\psi([\Ucal_2, \Tcal_2])\psi([\Vcal_2, \Scal_2])}.\]

Conversely, let $[\Ucal_1', \Tcal_1']$ and $[\Ucal_2', \Tcal_2']$ be $\tau$-perpendicular intervals of $\tors A$ such that $(\Ucal_1')^\perp \cap \Tcal_1' = (\Ucal_2')^\perp \cap \Tcal_2'$, then \cref{thm:wideintiso} and
\cref{prop:brickpreserve} imply that there exists a brick-label preserving isomorphism $\varphi': [\Ucal_1', \Tcal_1'] \to [\Ucal_2', \Tcal_2']$ which lifts to a join-irreducible-preserving isomorphism
\[ \phi_{[\Ucal_1', \Tcal_1'][\Ucal_2', \Tcal_2']}' \coloneqq \psi^{-1} \varphi' \psi: \psi^{-1}([\Ucal_1', \Tcal_1')] \to \psi^{-1}([\Ucal_2' , \Tcal_2']). \]

From \cref{lem:labelcoincide} it follows that $\jirrc(\psi^{-1}([\Ucal_1', \Tcal_1'])) = \jirrc(\psi^{-1}([\Ucal_1', \Tcal_1']))$. Therefore 
\[ [\Ucal_1', \Tcal_1'] \sim_{\Tfrak(A)} [\Ucal_2', \Tcal_2'] \quad \Longrightarrow \quad \psi^{-1}([\Ucal_1', \Tcal_1']) \sim_{\Tfrak(L)} \psi^{-1}([\Ucal_2', \Tcal_2']). \]

Consequently the objects of $\Tfrak(A)$ and $\Tfrak(L)$ are in bijection. Now let $[\Vcal_i', \Scal_i']$ be $\tau$-perpendicular intervals such that $[\Vcal_i', \Scal_i']  \subseteq [\Ucal_i', \Tcal_i']$ for $i=1,2$. Then $f_{[\Ucal_1', \Tcal_1'] [\Vcal_1', \Scal_1']} \sim_{\Tfrak(A)} f_{[\Ucal_2', \Tcal_2'][\Vcal_2', \Scal_2']}$ means $[\Vcal_1', \Scal_1'] \cap ((\Ucal_1')^\perp \cap \Tcal_1') = [\Vcal_2', \Scal_2'] \cap ((\Ucal_2')^\perp \cap \Tcal_2')$. Since $\varphi'$ factors through $\tors ((\Ucal_1')^\perp \cap \Tcal_1') = \tors ((\Ucal_2')^\perp \cap \Tcal_2')$, it follows that 
\[ \phi_{[\Ucal_1', \Tcal_1'][\Ucal_2', \Tcal_2']}([\Vcal_1', \Scal_1'] ) = [\Vcal_2', \Scal_2']. \]
Therefore $\phi_{[\Ucal_1', \Tcal_1'][\Ucal_2', \Tcal_2']}' \psi^{-1}([\Vcal_1', \Scal_1']) = ([\Vcal_2', \Scal_2'])$, which imples 
\[ f_{[\Ucal_1', \Tcal_1'][\Vcal_1', \Scal_1']} \sim_{\Tfrak(A)} f_{[\Ucal_2', \Tcal_2'][\Vcal_2', \Scal_2']} \quad \Longrightarrow \quad f_{\psi^{-1}([\Ucal_1', \Tcal_1'])\psi^{-1}([\Vcal_1', \Scal_1'])} \sim_{\Tfrak(L)} f_{\psi^{-1}([\Ucal_2', \Tcal_2'])\psi([\Vcal_2', \Scal_2'])}. \]
Hence morphisms sets of $\Tfrak(A)$ and $\Tfrak(L)$ are in bijection and the construction of the $\Tfrak(A)$ and $\Tfrak(L)$ coincide. As a consequence, the categories $\Tfrak(A)$ and $\Tfrak(L)$ are equivalent.
\end{proof}

As an immediate consequence we get the following.

\begin{cor}\label{cor:equivcats}
    Let $A$ be $\tau$-tilting finite and $B$ be $\tau$-tilting equivalent to $A$. Then there exists an equivalence of categories $\Tfrak(A) \cong \Tfrak(B)$.
\end{cor}
\begin{proof}
    Since $A$ is $\tau$-tilting finite $\tors A = \ftors A \cong \ftors B$ and therefore $B$ is $\tau$-tilting finite. Then both are isomorphic to an abstract finite lattice $L \cong \ftors A$ and we get $\Tfrak(A) \cong \Tfrak(L) \cong \Tfrak(B)$ from \cref{thm:categoryiso}.
\end{proof}

As an alternative definition, see \cite[Prop. 11.7]{BuanMarsh2021}, \textit{signed $\tau$-exceptional sequences} can be defined as factorisations of morphisms in $\Tfrak(A)$ into irreducible morphisms. As a consequence of \cref{thm:categoryiso}, we get the following. 
\begin{cor}\label{cor:signedexcep}
    Let $A$ be $\tau$-tilting finite and $B$ be $\tau$-tilting equivalent to $A$. Then there exists a bijection  
    \[ \{ \text{signed $\tau$-exceptional sequences of $\mods A$}\} \longleftrightarrow \{ \text{signed $\tau$-exceptional sequences of $\mods B$}\}. \]
\end{cor}
\begin{proof}
    The fully faithful functor of the equivalence of categories $\Tfrak(A) \cong \Tfrak(B)$ of \cref{cor:equivcats} induces a bijection between $\Hom$-sets of $\Tfrak(A)$ and $\Tfrak(B)$, which consequently gives a bijection between all factorisations of morphisms in the $\Hom$-sets and thus all signed $\tau$-exceptional sequences corresponding with those.
\end{proof}

\begin{rmk}
    Recently, a mutation for $\tau$-exceptional sequences was defined \cite{BHM2024}. This mutation generalises that of exceptional sequences of hereditary algebras \cite{cbw92, ringel_exceptional}. if $A$ and $B$ are two finite-dimensional algebras such that $\tors A \cong \tors B$ and both are finite, then there is a bijection between $\tau$-exceptional sequences of $\mods A$ and $\mods B$ by \cite[Thm. 8.10, Rmk. 8.11]{BarnardHanson2022exc}. However, the mutation of $\tau$-exceptional sequences heavily relies on signed $\tau$-exceptional sequences. Since \cref{cor:signedexcep} shows that these are in bijection as well, it seems plausible that the mutation of $\tau$-exceptional sequences may be determined by the lattice of torsion classes. This could provide further inside into the problem of transitivity of the mutation, the understanding of which is still limited to two special cases, see \cite{BHM2024, BKT2025}.
\end{rmk}

\section{Factor algebras and lattice congruences} \label{sec:quotients}
We now use the lattice-theoretic definition of the $\tau$-cluster morphism category to gain new insights into its structure and its properties. In particular, as illustrated in \cref{exmp:good} this approach appears to be the most natural one to study factor algebras of $A$ by an ideal $I$. An intuitive reason for this is that $\tors A \cap \mods A/I = \tors A/I$ but the same behaviour is generally not exhibited by $\tau$-rigid pairs or wide subcategories. More precisely, recall from \cref{thm:DIRRTquot} that an ideal $I \in \ideal A$ induces a surjective morphism of lattices $\tors A \twoheadrightarrow \tors A/I$ given by $\Tcal \mapsto \Tcal \cap \mods A/I$. The following result describes the lattice congruence $\Phi_I$ on $\tors A$ in terms of the brick labelling.

\begin{thm} \cite[Thm. 5.15]{DIRRT2017} \label{thm:brickcontract}
    Let $A$ be a finite-dimensional algebra and $I$ an ideal. 
    \begin{enumerate}
        \item An arrow $q \in \Hasse(\tors A)$ is not contracted by $\Phi_I$ if and only if its brick label $B_q$ is in $\mods A/I$. In this case, it has the same label in $\Hasse(\tors A)$ and $\Hasse(\tors A/I)$.
        \item Let $\Ucal \subseteq \Tcal \in \tors A$, then $\Ucal \equiv_{\Phi_I} \Tcal$ if and only if, $IB \neq 0$ for all bricks $B \in \Ucal^\perp \cap \Tcal$.
    \end{enumerate}
\end{thm}

To understand the behaviour of intervals under the surjective lattice morphism $\tors A \twoheadrightarrow \tors A/I$ given by intersecting with $\mods A/I$ we need the following two immediate results. Throughout the following, fix an ideal $I \in \ideal A$.

\begin{lem} \label{lem:perpinquot}
    Let $\Ucal \in \tors A$. We have $(\Ucal \cap \mods A/I)^{\perp_{A/I}} = \Ucal^{\perp_A} \cap \mods A/I$.
\end{lem}
\begin{proof}
    To show the inclusion $\supseteq$, take $X \in \Ucal^{\perp_{A}} \cap \mods A/I$. Assume for a contradiction, there exists $Y \in \Ucal \cap \mods A/I$ such that $\Hom_{A/I}(Y,X) \neq 0$. Since $Y \in \Ucal$ it follows immediately that $X \not \in \Ucal^{\perp_A}$, a contradiction. 
    
    To show $\subseteq$, take $X \in ( \Ucal \cap \mods A/I)^{\perp_{A/I}} \subseteq \mods A/I$. We need to show that $X \in \Ucal^{\perp_A}$. Assume for a contradiction that there exists $0 \neq f \in \Hom_A(Y,X)$ with $Y \in \Ucal$. By the assumption on $X$, it follows that $Y \not \in \mods A/I$. However, $\image f$ is a submodule of $X$, and thus lies in $\mods A/I$. Moreover, since $\Ucal$ is a torsion class and thus closed under quotients, we have $\image f \in \Ucal \cap \mods A/I$.  Hence there should not be a non-zero morphism $\image f \to X$. Since $f \neq 0$ implies $\image f \neq 0$, this is a contradiction and we have shown equality.
\end{proof}

\begin{cor} \label{cor:wideinquot}
    Let $[\Ucal, \Tcal] \subseteq \tors A$. We have $\overline{\Ucal}^{\perp_{A/I}} \cap \overline{\Tcal} = \Ucal^{\perp_A} \cap \Tcal \cap \mods A/I$.
\end{cor}
\begin{proof}
    Clearly $\overline{\Ucal}^{\perp_{A/I}} \cap \overline{\Tcal} = (\Ucal \cap \mods A/I)^{\perp_A} \cap (\Tcal \cap \mods A/I) = \Ucal^{\perp_A} \cap \Tcal \cap \mods A/I$ by \cref{lem:perpinquot}.
\end{proof}

The following is a simple observation but since wide subcategories play a central role in this paper we include it for convenience. 
\begin{lem}\label{lem:widetowide}
    Let $\Wcal$ be a wide subcategory of $\mods A$, then $\Wcal \cap \mods A/I$ is a wide subcategory of $\mods A/I$.
\end{lem}
\begin{proof}
    Let $L,M,N \in \mods A/I$ lie in a short exact sequence $0 \to L \to M \to N \to 0$ with $L, N \in \Wcal \cap \mods A/I \subseteq \mods A$. Since $L,N \in \Wcal \subseteq \mods A$, we get $M \in \Wcal$ since $\Wcal$ is wide. Hence $M \in \Wcal \cap \mods A/I$ as required. Similarly, let $f: M \to N$ be a morphism with $M,N \in \Wcal \cap \mods A/I$, then since $M,N \in \Wcal$, we have $\ker f, \coker f \in \Wcal$. Moreover $\ker f\in \mods A/I$ and $\coker f \in \mods A/I$ since $\mods A/I$ is a full abelian subcategory of $\mods A$.
\end{proof}

\begin{cor}\label{cor:wideinttowideint}
    Let $[\Ucal, \Tcal]$ be a wide interval of $\tors A$, then $\overline{[\Ucal, \Tcal]}$ is a wide interval of $\tors A/I$.
\end{cor}
\begin{proof}
    By assumption $\Ucal^{\perp_A} \cap \Tcal$ is a wide subcategory of $\mods A$, then by \cref{cor:wideinquot} we have $\overline{\Ucal}^{\perp_{A/I}} \cap \overline{\Tcal} = \Ucal^{\perp_A} \cap \Tcal \cap \mods A/I$ which is a wide subcategory of $\mods A/I$ by \cref{lem:widetowide}.
\end{proof}

The following is the starting point for connecting $\Tfrak(A)$ and $\Tfrak(A/I)$.

\begin{prop} \label{lem:tauperptotauperp}
    If $[\Ucal, \Tcal]$ is a $\tau$-perpendicular interval of $\tors A$, then $\overline{[\Ucal, \Tcal]}$ is a $\tau$-perpendicular interval of $\tors A/I$.
\end{prop}
\begin{proof}
    Let $[\Ucal, \Tcal]$ be a $\tau$-perpendicular interval, in particular, it is wide. By \cref{cor:wideinttowideint} we have that $\overline{[\Ucal, \Tcal]}$ is a wide interval of $\tors A/I$. By \cite[Thm. 1.3]{DIRRT2017} and \cref{exmp:completelattice} the lattice $\tors A$ is a complete lattice and by \cref{thm:DIRRTquot} the lattice congruence $\Phi_I$ is a complete lattice congruence. By \cref{prop:quotlattiso}, we have $\pi_{\downarrow}^{\Phi_I} (\tors A) \cong \tors / \Phi_I \cong \tors A/I$. By the same result, congruence classes of $\Phi_I$ are intervals and in this case \cite[9.47, Prop. 9-5.8]{Reading2016} implies that $\pi_{\downarrow}^{\Phi_I} (\tors A)$ is a join-subsemilattice of $\tors A$.\\
    
    By \cite[Rmk. 9-5.9]{Reading2016} this means that $\pi_{\downarrow}^{\Phi_I} (x \lor y) = \pi_{\downarrow}^{\Phi_I} x \lor \pi_{\downarrow}^{\Phi_I} y$ holds. Since $[\Ucal, \Tcal]$ is $\tau$-perpendicular, by \cref{prop:DIRRT420} there are $s$ arrows $\{\Ucal_i \to \Ucal\}_{i=1}^s$ in $\Hasse([\Ucal, \Tcal])$ such that $\Tcal = \Ucal_1 \lor \dots \lor \Ucal_s$. From the join-subsemilattice property, we obtain 
    \begin{align*}
        \overline{\Tcal} &\cong \pi_{\downarrow}^{\Phi_I}(\Tcal) = \pi_{\downarrow}^{\Phi_I}(\Ucal_1 \lor \dots \lor \Ucal_s)  = \pi_{\downarrow}^{\Phi_I} \Ucal_1 \lor \dots \lor \pi_{\downarrow}^{\Phi_I} \Ucal_r  \cong \overline{\Ucal}_1 \lor \dots \lor \overline{\Ucal}_r.
    \end{align*} 
    Since $\overline{\Ucal} \in \ftors A/I$ by \cite[Prop. 5.6b]{DIRRT2017}, \cref{prop:DIRRT420} implies that the interval $[\overline{\Ucal}, \overline{\Tcal}]$ is $\tau$-perpendicular as required.
\end{proof}

This implies that we may extend the lattice isomorphisms between the $\tau$-perpendicular intervals of \cref{thm:wideintiso} to their quotients in a natural way.

\begin{prop} \label{wideintisoquot} We have the following commutative diagram extending \cref{thm:wideintiso}.
    \[ \begin{tikzcd}[column sep=5]
        {[\Ucal_{(M_1,P_1)}, \Tcal_{(M_1,P_1)}]} \arrow[rd,"{(-) \cap \Wcal_{(M_1,P_1)}}",swap] \arrow[rr, "\cong"]\arrow[dd, crossing over, "(-) \cap \mods A/I", swap]& & {[\Ucal_{(M_2,P_2)}, \Tcal_{(M_2,P_2})]} \arrow[ld, "{(-) \cap \Wcal_{(M_2,P_2)}}"] \arrow[dd, crossing over, "(-) \cap \mods A/I"]\\ 
        & \tors \Wcal_{(M_1, P_1)} = \tors \Wcal_{(M_2,P_2)} \arrow[dd, crossing over, "(-) \cap \mods A/I" pos=0.25] \\
        {\overline{[\Ucal_{(M_1,P_1)}, \Tcal_{(M_1,P_1)}]}} \arrow[rd,"{(-) \cap (\Wcal_{(M_1,P_1)} \cap \mods A/I)}",swap] \arrow[rr, "\cong" pos=0.3]& & {\overline{[\Ucal_{(M_2,P_2)}, \Tcal_{(M_2,P_2)}]}} \arrow[ld, "{(-) \cap (\Wcal_{(M_2,P_2)} \cap \mods A/I)}"]\\ 
        & \tors (\Wcal_{(M_i, P_i)} \cap \mods A/I)
    \end{tikzcd}
    \]
    The top and bottom parts are lattice isomorphisms coming from \cref{thm:wideintiso}, and the downward-facing arrows are given by intersecting with $\mods A/I$.
\end{prop}
\begin{proof}
Let $\Wcal \coloneqq \Wcal_{(M_1,P_1)} = \Wcal_{(M_2,P_2)}$. By \cref{lem:widetowide}, $\Wcal \cap \mods A/I$ is a wide subcategory of $\mods A/I$. For $\Tcal \in \tors \Wcal$ it follows trivially that $\Tcal \cap \mods A/I$ is a torsion class of the wide subcategory $\Wcal \cap \mods A/I$ of $\mods A/I$. By \cref{lem:tauperptotauperp} the intervals 
\[ {\overline{[\Ucal_{(M_1,P_1)}, \Tcal_{(M_1,P_1)}]}} \text{ and } {\overline{[\Ucal_{(M_2,P_2)}, \Tcal_{(M_2,P_2)}]}}\]
are $\tau$-perpendicular intervals of $\tors A/I$ and correspond to the wide subcategory $\Wcal \cap \mods A/I = \Wcal \cap \mods A/I$ by \cref{cor:wideinquot}. Thus \cref{thm:wideintiso} implies the existence of the lattice isomorphisms in the bottom half of the diagram and the downward-facing arrows are well-defined. The commutativity of the squares is obvious from the description of the maps. 
\end{proof}

We are now able to establish a functor between the categories.

\begin{thm} \label{thm:inducedfunctor}
    Let $I \in \ideal A$. There exists a functor $F_I: \Tfrak(A) \to \Tfrak(A/I)$ induced by $\overline{(-)}_I$, that is, $F_I$ is given on objects by $[\Ucal, \Tcal]_\sim \mapsto \overline{[\Ucal, \Tcal]}_\sim$ and on morphisms by $[f_{[\Vcal, \Scal][\Ucal, \Tcal]}] \mapsto [f_{\overline{[\Vcal, \Scal]}\overline{[\Ucal, \Tcal]}}]$.
\end{thm}
\begin{proof}
    $F_I$ maps $\tau$-perpendicular intervals to $\tau$-perpendicular intervals by \cref{lem:tauperptotauperp} which also implies that $F_I$ is well-defined on objects since
    \begin{align*} [\Ucal_1, \Tcal_1] \sim [\Ucal_2, \Tcal_2] &\Leftrightarrow \Ucal_1^\perp \cap \Tcal_1 = \Ucal_2^\perp \cap \Tcal_2 \\
    &\Rightarrow \Ucal_1^\perp \cap \Tcal_1 \cap \mods A/I= \Ucal_2^\perp \cap \Tcal_2 \cap \mods A/I \\
    &\Leftrightarrow \overline{[\Ucal_1, \Tcal_1]} \sim \overline{[\Ucal_2, \Tcal_2]}
    \end{align*}
    by \cref{cor:wideinquot}. It is clear that containment of intervals is preserved by $\overline{(-)}_I$ because intersection with $\mods A/I$ is order-preserving. Let $[\Ucal_i, \Tcal_i] \subseteq [\Vcal_i, \Scal_i]$ for $i =1,2$ be such that $f_{[\Vcal_1, \Scal_1][\Ucal_1, \Tcal_1]} \sim_{\Tfrak(A)} f_{[\Vcal_2, \Tcal_2][\Ucal_2, \Tcal_2]}$. The commutativity of the diagram in \cref{wideintisoquot} implies that 
   \begin{align*}
    	&[\Ucal_1, \Tcal_1] \cap (\Vcal_1^\perp \cap \Scal_1) = [\Ucal_2, \Tcal_2] \cap (\Vcal_2^\perp \cap \Scal_2)  \\
	&\Longrightarrow \overline{[\Ucal_1, \Tcal_1]} \cap (\Vcal_1^\perp \cap \Scal_1 \cap \mods A/I) = \overline{[\Ucal_2, \Tcal_2]} \cap (\Vcal_2^\perp \cap \Scal_2 \cap \mods A/I).
    \end{align*}
    
    Therefore, $F_I$ is well-defined on morphisms i.e. $f_{[\overline{\Vcal}_1, \overline{\Scal}_1][\overline{\Ucal}_1, \overline{\Tcal}_1]} \sim_{\Tfrak(A/I)} f_{[\overline{\Vcal}_2, \overline{\Tcal}_2][\overline{\Ucal}_2, \overline{\Tcal}_2]}$. To show hat composition is preserved it is sufficient, by \cref{lem:composition1}, to consider three $\tau$-perpendicular intervals $[\Xcal , \Ycal] \subseteq [\Ucal, \Tcal] \subseteq [\Vcal, \Scal] \subseteq \tors A$. The following is obvious:
    \begin{align*} 
    F_I([f_{[\Ucal, \Tcal][\Xcal, \Ycal]}] \circ [f_{[\Vcal, \Scal][\Ucal, \Tcal]}]) & = F_I([f_{[\Vcal, \Scal][\Xcal, \Ycal]}]) \\
    &= [f_{\overline{[\Vcal, \Scal][\Ucal, \Tcal]}}] \\
    &= [f_{\overline{[\Ucal, \Tcal][\Xcal, \Ycal]}}] \circ [f_{\overline{[\Vcal, \Scal][\Ucal, \Tcal]}}] \\
    &= F_I([f_{[\Ucal, \Tcal][\Xcal, \Ycal]}]) \circ F_I([f_{[\Vcal, \Scal][\Ucal, \Tcal]}]).
    \end{align*} 
    It is clear that identity morphisms are preserved. 
\end{proof}

We are interested in understanding the categorical properties of the functor $F_I$. 

\begin{lem} \label{lem:FIfaithfuliff}
    The functor $F_I$ is faithful if and only if the restriction $\overline{(-)}: \ftors A \to \ftors A/I$ is injective.
\end{lem}
\begin{proof}
    $(\Leftarrow)$. Let $[f_{[\Vcal, \Scal][\Ucal_1, \Tcal_1]}], [f_{[\Vcal, \Scal] [\Ucal_2, \Tcal_2]}] \in \Hom_{\Tfrak(A)}([\Vcal, \Scal]_\sim, [\Ucal, \Tcal]_\sim)$ be distinct morphisms, which is to say that either $\Ucal_1 \neq \Ucal_2$ or $\Tcal_1 \neq \Tcal_2$. Then applying $F_I$ gives
    $[f_{\overline{[\Vcal, \Scal]}, \overline{[\Ucal_1, \Tcal_1]}}]$ and $[f_{\overline{[\Vcal, \Scal]}, \overline{[\Ucal_2, \Tcal_2]}}]$ in $\Hom_{\Tfrak(A/I)}({\overline{[\Vcal, \Scal]}_{\sim}, \overline{[\Ucal, \Tcal]}}_{\sim})$,
    which coincide if and only if $\overline{[\Ucal_1, \Tcal_1]}= \overline{[\Ucal_2, \Tcal_2]}$. So both $\overline{\Ucal}_1 =\overline{\Ucal}_2$, $\overline{\Tcal}_1 = \overline{\Tcal}_2$. which would imply that $\overline{(-)}: \ftors A \to \ftors A/I$ is not injective. Thus $F_I$ is injective on $\Hom$-sets, hence faithful. \\
    $(\Rightarrow)$. Let $F_I$ be faithful and assume for a contradiction that $\overline{(-)}: \ftors A \to \ftors A/I$ is not an injection.  
    There exists a distinct morphism $[f_{[0, \mods A][\Tcal,\Tcal]}]: [0, \mods A]_\sim \to  [\Tcal, \Tcal]_\sim $ for every functorially-finite torsion class $\Tcal \in \ftors A$, see also \cref{rmk:latticecapturing}. If $\overline{(-)}$ is not injective, then $\Tcal_1 \cap \mods A/I = \Tcal_2 \cap \mods A/I$ for two distinct $\Tcal_1, \Tcal_2 \in \ftors A$. However, then the distinct morphisms 
    \[ [f_{[0, \mods A][\Tcal_1, \Tcal_1]} ]\neq  [ f_{[0,\mods A][\Tcal_2, \Tcal_2]} ] \]
   of $\Tfrak(A)$ have the same image under $F_I$ and hence $F_I$ would not be faithful.  
\end{proof}

\begin{rmk}
    If $A$ is $\tau$-tilting finite, we have $\tors A = \ftors A$ by \cite[Thm. 3.8]{DIJ2019} and thus $\overline{(-)}: \ftors A \to \ftors A/I$ is surjective by \cite[Prop. 5.7d)]{DIRRT2017}. In this case \cref{lem:FIfaithfuliff} holds true if there is a bijection and thus a lattice isomorphism $\tors A \cong \tors A/I$.
\end{rmk}

Before continuing to study properties of $F_I$ we need the following classical relationship between wide subcategories and semibricks. We say that a module $B = B_1 \oplus \dots \oplus B_r$ is a \textit{semibrick} if $B_i$ are bricks and $\Hom(B_i, B_j)=0$ for all $i \neq j$. The second property will be referred to as the bricks being pairwise $\Hom$-orthogonal.

\begin{thm}\cite{Ringel1976} \label{thm:Ringelbij} 
    Let $B_1, \dots, B_k$ be pairwise $\Hom$-orthogonal bricks in $\mods A$. Then there is a bijection
    \begin{align*}
        \wide A &\leftrightarrow \sbrick A\\
        \Wcal & \mapsto \bigoplus_{B_i \in \simp \Wcal} B_i \\
        \Filt_A \{ B_1, \dots, B_r\} &\mapsfrom B_1 \oplus \dots \oplus B_r
    \end{align*} 
    wide subcategories and semibricks, where $\simp \Wcal$ is the collection of (relative) simple modules of $\Wcal$.
\end{thm}

For any $I \in \ideal A$, define the following map using \cref{thm:Ringelbij}:
\begin{equation} \label{eq:wideinclusion}
    \begin{aligned}
        \iota: \wide A/I &\to \wide A \\
        \Filt_{A/I} \{ B_1, \dots, B_r\} &\mapsto \Filt_{A} \{ B_1, \dots, B_r\}
    \end{aligned}
\end{equation}

When the algebra $A$ is $\tau$-tilting finite, then every wide subcategory is $\tau$-perpendicular \cite[Cor. 2.17]{IngallsThomas2009}, see also \cite[Rmk. 4.10]{BuanHanson2017} and $A/I$ is $\tau$-tilting finite \cite[Cor. 1.9]{DIRRT2017}. In this setting we can lift semibricks, which in turn allows us to lift $\tau$-perpendicular subcategories. This way, we are able to show the following.

\begin{lem} \label{prop:surjonobj}
    Let $A$ be $\tau$-tilting finite and $I \in \ideal A$, then the functor $F_I: \Tfrak(A) \to \Tfrak(A/I)$ in \cref{thm:inducedfunctor} is surjective-on-objects.
\end{lem}
\begin{proof}
    Let $[\Ucal, \Tcal]_\sim \in \Tfrak(A/I)$, then $\Ucal^{\perp_{A}} \cap \Tcal = \Filt_{A/I} \{ B_1, \dots, B_r\}$ for some semibrick $B_1 \oplus \dots \oplus B_r \in \mods A/I \subseteq \mods A$, by \cref{thm:Ringelbij}. Consider the lifted wide subcategory $\iota (\Filt_{A/I} \{ B_1, \dots, B_r\}) \subseteq \mods A$. Since $A$ is $\tau$-tilting finite, $\Filt_A \{ B_1, \dots, B_r\}$ is also a $\tau$-perpendicular subcategory. By \cite[Thm. 4.5]{BuanHanson2017} there exists some $\tau$-perpendicular interval $[\Acal, \Bcal] \subseteq \tors A$ such that $\Acal^{\perp_A} \cap \Bcal = \Filt_A \{ B_1, \dots, B_r\}$. It is clear that $\overline{[\Acal, \Bcal]} \sim [\Ucal, \Tcal]$ since, using \cref{lem:perpinquot} to obtain the first equality, we have
    \[ \overline{\Acal}^{\perp_{A/I}} \cap \overline{\Bcal} = \Acal^{\perp_A} \cap \Bcal \cap \mods A/I = \Filt_{A} \{ B_1, \dots, B_r\} \cap \mods A/I = \Filt_{A/I} \{ B_1, \dots, B_r\} = \Ucal^{\perp_A} \cap \Tcal. \]
    Hence every object $[\Ucal, \Tcal]_\sim \in \Tfrak(A/I)$ lies in the image of $F_I$.
\end{proof}

We illustrate in \cref{exmp:taufinite} that the assumption of $\tau$-tilting finiteness is necessary. To understand in which cases $F_I: \Tfrak(A) \to \Tfrak(A/I)$ is full we first need to introduce the Cartesian product of posets.

\begin{defn}
    Let $(P_1, \leq_1)$ and $(P_2, \leq_2)$ be two posets. Define their product $(P_1 \times P_2, \leq)$ via the partial order
    \[ (a,b) \leq (c,d) \quad \Leftrightarrow \quad a \leq_1 c \text{ and } b \leq_2 d.\]
    If $(P_1, \leq_1)$ and $(P_2, \leq_2)$ are lattices, define the join and meet component-wise, then $(P_1 \times P_2, \leq)$ is a lattice.
\end{defn}

We now prove the following intermediate result. See \cite[Thm. 4.19a]{AHIKM2022} for a similar result about $g$-vector fans.
\begin{lem} \label{lem:product}
    Let $A  \cong A_1 \times A_2$ where $A_1, A_2$ are finite-dimensional algebras. Then as posets we have 
    \[ \tors A \cong \tors A_1 \times \tors A_2, \quad \text{and} \quad \ftors A \cong \ftors A_1 \times \ftors A_2. \]
\end{lem}
\begin{proof}
     It is well-known that we have a $\mods A \cong \mods A_1 \times \mods A_2$. This implies that the inverse bijections are given by
    \begin{align*}
        \tors A &\cong \tors A_1 \times \tors A_2 \\
        \Tcal &\mapsto (\Tcal \cap \mods A_1, \Tcal \cap \mods A_2) \\
        \add(\Tcal_1 \cup \Tcal_2) & \mapsfrom (\Tcal_1, \Tcal_2).
    \end{align*}
   If $\Tcal \in A$ is a torsion class, then it is clear that $\Tcal \cap \mods A_i \in \tors A_i$ for $i = 1,2$ since both terms of the intersection are closed under extensions and quotients. In particular, if $\Tcal$ is functorially-finite then $\Tcal \cap \mods A_I$ is functorially-finite, for $i=1,2$ by \cite[Prop. 5.6b]{DIRRT2017}. Similarly, $\add(\Tcal_1 \cup \Tcal_2)$ is a torsion class, since $\Ext_A^1(X_1^a, X_2^b) =0 = \Ext_A^1(X_2^b, X_1^d)$ for all $X_1 \in \mods A_1$, $X_2 \in \mods A_2$ and $a,b,c,d \geq 1$ implies closure under extensions. Because $\Hom_A(X_1,X_2) =0= \Hom_A(X_2,X_1)$ it follows that $\add(\Tcal_1 \cup \Tcal_2)$ is functorially-finite, since any module $M \in \mods A$ admits left and right $\add(\Tcal_1 \cup \Tcal_2)$-approximations, which are simply given by a $\Tcal_1$-approximation of the direct summand of $M$ which is in $\mods A_1$, and a $\Tcal_2$-approximation of the direct summand of $M$ which is in $\mods A_2$. One sees directly that these are inverse assignments and order-preserving.
\end{proof}

For remainder of this section, assume for simplicity that the field $K$ is algebraically closed. This assumption allows us to say that $A \cong B \times C$ for some finite-dimensional algebras $B$ and $C$ if there exist two sets of simple $A$-modules $\Scal_1$ and $\Scal_2$ such that $\Ext_A^1(S_1, S_2) = \Ext_A^1(S_2, S_1)$ for $S_i \in \Scal_i$ and $i=1,2$. This is because we may write $A \cong KQ/I$ for some ideal $I$, where $Q$ is the $\Ext$-quiver of the algebra, see \cite[Lem. II.2.5, Lem. III.2.12]{bluebookV1}. As a converse to \cref{lem:product} we have the following. 

\begin{lem}\label{lem:productconverse}
    Let $A$ be a finite-dimensional algebra such that $\tors A \cong \tors B' \times \tors C'$, for some finite-dimensional algebras $B'$ and $C'$. Then $A \cong B \times C$ for some finite-dimensional algebras $B$ and $C$.
\end{lem}
\begin{proof}
    Denote the lattice isomorphism by $\phi: \tors B' \times \tors C' \to \tors A$. Take $X_1$ to be a simple $B'$-module and $X_2$ to be a simple $C'$-module, then by definition of the join we have 
    \[ (\Filt_{B'}(X_1),0) \lor (0, \Filt_{C'}(X_2) = (\Filt_{B'}\{X_1\}, \Filt_{C'}\{X_2\}). \]
    Applying $\phi$ to both sides of the equation, we deduce that there are simple $A$-modules, say $\widehat{X}_1$ and $\widehat{X}_2$, labelling the cover relations $\phi(\Filt_{B'}(X_1), 0) \to 0 \subseteq \Hasse(\tors A)$ and $\phi(0,\Filt_{C'}(X_2)) \to 0 \subseteq \Hasse(\tors A)$ respectively, such that $\Filt_A\{ \widehat{X}_1, \widehat{X}_2\}$ contains no bricks other than $B_1$ and $B_2$. It follows that $\Ext_A^1(\widehat{X}_1, \widehat{X}_2)=0=\Ext_A^1(\widehat{X}_2, \widehat{X}_1)$ by \cite[Lem. 4.26]{DIRRT2017}. Thus there exist two sets of simple modules $\Scal_1, \Scal_2 \subseteq \mods A$ such that $\Scal_1 \cup \Scal_2 = \simp(\mods A)$ and such that $\Ext_A^1(S_1, S_2) = 0 = \Ext_A^1(S_2,S_1)$ for $S_i \in \Scal_i$ for $i=1,2$. Define 
    \[ B \cong A/\left\langle \sum_{i : S(i) \in \Scal_1} e_i\right\rangle, \quad \text{and} \quad C \cong A/\left\langle \sum_{i : S(i) \in \Scal_2} e_i\right\rangle. \]
    Then $A \cong B \times C$ as required.
\end{proof}

We can now describe when the functor $F_I: \Tfrak(A) \to \Tfrak(A/I)$ is full. To avoid further technicalities involving isomorphisms between infinite lattices of functorially-finite torsion classes like in \cref{exmp:tauequiv} (3), we restrict ourselves to the finite case for the following result.

\begin{prop}\label{lem:Ffulliff}
    Let $A$ be $\tau$-tilting finite. The functor $F_I$ is full if and only if $\tors A \cong \tors A/I \times \tors B$ for some finite-dimensional algebra $B$. 
\end{prop}
\begin{proof}
    $(\Leftarrow)$. It follows from (the proof of) \cref{lem:productconverse} that we can write $A \cong C \times B$ where $C$ is a quotient algebra of $A$ by an ideal generated by the primitive orthogonal idempotent such that $\tors C \cong \tors A/I$. By \cref{thm:categoryiso} we have an equivalence $G: \Tfrak(C) \xrightarrow{\cong} \Tfrak(A/I)$. Moreover, we have a sequence of surjective algebra morphisms $A \twoheadrightarrow C \twoheadrightarrow A/I$ and hence the lattice congruences on $\tors A$ induced by $A \twoheadrightarrow C$ and $A \twoheadrightarrow A/I$ coincide. Furthermore, the functor $F_I$ factors through the equivalence, i.e. $F_I = G \circ F_C$, where $F_C: \Tfrak(A) \to \Tfrak(C)$ is induced by the surjection $A \twoheadrightarrow C$. Since the composition of full functors is full, it is sufficient to show that $F_C$ is full. \\

    Let $[\Vcal', \Scal']_\sim, [\Ucal', \Tcal']_\sim \in \Tfrak(A)$ be such that there is a pair of representatives $[\Vcal, \Scal] \in [\Vcal', \Scal']_\sim$ and $[\Ucal, \Tcal] \in [\Ucal', \Tcal']_\sim$ such that $[\Vcal, \Scal] \supseteq [\Ucal, \Tcal]$. We need to show that the induced map
    \[ \Hom_{\Tfrak(A)}([\Vcal, \Scal]_\sim, [\Ucal, \Tcal]_\sim) \to \Hom_{\Tfrak(C)} ([\Vcal \cap \mods C, \Scal \cap \mods C]_\sim, [\Ucal \cap \mods C, \Tcal \cap \mods C]_\sim)\]
    is surjective. Thus take an arbitrary morphism $f_{[\Vcal_1, \Scal_1][\Ucal_1, \Tcal_1]} \in \Tfrak(C)$ in the codomain, where $[\Vcal_1, \Scal_1] \sim [\Vcal \cap \mods C, \Scal \cap \mods C]$ and $[\Ucal_1, \Tcal_1]\sim [\Ucal \cap \mods C, \Tcal \cap \mods C]$. Define 
    \[ [\widehat{\Ucal}, \widehat{\Tcal}] \coloneqq [ \add(\Ucal_1 \cup (\Ucal \cap \mods B)), \add(\Tcal_1 \cup (\Tcal \cap \mods B))],\]
    which is an interval of $\tors A$ because both terms are additive closures of a torsion classes in $\mods B$ and a torsion class in $\mods C$. Since $\mods A \cong \mods B \times \mods C$ the result is a torsion class, similar to the proof of \cref{lem:product}.Then, the simple objects of the wide subcategories satisfy 
    \begin{align*}
        \simp (\Ucal^{\perp_A} \cap \Tcal) & = (\simp(\Ucal^{\perp_A} \cap \Tcal) \cap \mods C) \cup (\simp(\Ucal^{\perp_A} \cap \Tcal) \cap \mods B)\\
        & = \simp (\Ucal_1^{\perp_C} \cap \Tcal_1 ) \cup (\simp(\Ucal^{\perp_A} \cap \Tcal \cap \mods B)) \\
        & = \simp (\widehat{\Ucal}^{\perp_A} \cap \widehat{\Tcal}),
    \end{align*}
    where the second equality follows from \cref{cor:wideinquot}. By \cite[Thm. 4.16(a)]{DIRRT2017} this 
    implies $\Ucal^{\perp_A} \cap \Tcal = \widehat{\Ucal}^{\perp_A} \cap \widehat{\Tcal}$ and thus $[\widehat{\Ucal}, \widehat{\Tcal}]_\sim = [\Ucal, \Tcal]_\sim \in \Tfrak(A)$. Similarly we define
    \[ [\widehat{\Vcal}, \widehat{\Scal}] \coloneqq [\add(\Vcal_1 \cup (\Vcal \cap \mods B)), \add(\Scal_1 \cup (\Scal \cap \mods B))] \]
    and obtain $[\widehat{\Vcal}, \widehat{\Scal}]_\sim =  [\Vcal, \Scal]_\sim \in \Tfrak(A)$. Then, under $F_C$, the morphism $[f_{[\widehat{\Vcal}, \widehat{\Scal}][\widehat{\Ucal}, \widehat{\Tcal}]}] \in \Hom_{\Tfrak(A)}([\Vcal, \Scal]_\sim, [\Ucal, \Tcal]_\sim)$ gets mapped to $[f_{[\Vcal_1, \Scal_1][\Ucal_1, \Tcal_1]}] \in \Tfrak(C)$ as required. Therefore the induced map between $\Hom$-sets is surjective. Hence $F_C$ is full, and so is $F_I = G \circ F_C$. \\

    $(\Rightarrow)$. Assume $F_I$ is full. If $\tors A \cong \tors A/I$, we are done, otherwise let $\Tcal_1 \xleftarrow{B} \Tcal_2 \subseteq \Hasse(\tors A)$ be an arrow contracted by the lattice congruence. In particular since $F_I$ is full, the following induced map of $\Hom$-sets is surjective:
    \[ \Hom_{\Tfrak(A)}([0,\mods A]_\sim, [\Tcal_1, \Tcal_2]_\sim) \to \Hom_{\Tfrak(A/I)}([0,\mods A/I]_\sim, [0,0]_\sim). \]
    Let $\Xcal \in \tors A/I$ be any torsion class, then the corresponding morphism $[f_{[0,\mods A/I][\Xcal, \Xcal]}] \in \Tfrak(A/I)$ lies in the codomain of the map above. Thus the above map is surjective only if there exists an arrow $\Tcal_3 \xleftarrow{B} \Tcal_4 \in \Hasse([\pi_{\downarrow} \Xcal, \pi_{\uparrow} \Xcal])$. Applying this idea to all contracted arrows in a congruence class of $\Phi_I$, we deduce that 
    \begin{equation}\label{eq:samebricks} \brick [\pi_{\downarrow} \Xcal_1, \pi_{\uparrow} \Xcal_1] = \brick [\pi_{\downarrow} \Xcal_2, \pi_{\uparrow} \Xcal_2],\end{equation}
    for all $\Xcal_1, \Xcal_2 \in \tors A/I$. Let $\Scal_2 \coloneqq \simp A/I$ and $\Scal_1 \coloneqq \simp A \setminus \Scal_2$. We will show that $\Ext_A^1(S_1, S_2) = 0 = \Ext_A^1(S_2, S_1)$ for all simple modules $S_1 \in \Scal_1$ and $S_2 \in \Scal_2$. This implies that the corresponding idempotents $\epsilon_1 \coloneqq \sum_{i: S(i) \in \Scal_1} e_i$ and $\epsilon_2 \coloneqq \sum_{i: S(i) \in \Scal_2} e_i$ are central and thus that $A$ is not connected. Take $S_1 \in \Scal_1$ and $S_2 \in \Scal_2$. Since $A$ is $\tau$-tilting finite, description of the lattice congruence of \cite[Prop. 4.21, Thm. 4.23]{DIRRT2017} and \cite[Thm. 9-6.5]{Reading2016} imply that all arrows arising in $[0, \Filt_A\{ S_1, S_2\}] \subseteq \tors A$ except the two labelled by $S_2$ are contracted by the congruence $\Phi_I$. Hence this polygon consists of two halves, one half is the side lying in the congruence class $\Phi_I^{-1}(0) \subseteq \tors A$, the other lies in the congruence class $\Phi_I^{-1}(\Filt_{A/I}\{S_2\}) \subseteq \tors A$. The two halves are connected by the two arrows labelled by $S_2$. \\
    
    Let $0 \to S_2 \to M \to S_1 \to 0$ be a non-split short exact sequence in $\mods A$, so in particular $M \not \in S_2^{\perp_A}$. Now, $M$ is a brick by \cite[Lem. 4.26]{DIRRT2017} and arises as a label of an arrow in the polygon $[0, \Filt_A\{ S_1, S_2\}] \subseteq \tors A$ by \cite[Thm. 4.21(b)]{DIRRT2017}. More precisely $M$ is a label in the half of the polygon lying in $\Phi_I^{-1}(0) = [0, \pi_\uparrow 0]$ since $M \not \in S_2^{\perp_A}$. However, from \cref{eq:samebricks} it follows that $M$ must arise as a label of some arrow in $\Phi_I^{-1}(\Filt_{A/I}\{S_2\}) = [\Filt_A\{ S_2\}, \pi_\uparrow \Filt_{A/I}\{S_2\}]$, which is a contradiction, since this requries $M \in \Filt_A\{ S_2\}^{\perp_A} = S_2^{\perp_A}$ by definition of the brick-labelling. Thus $\Ext_A^1(S_1, S_2)=0$.\\

    Let $0 \to S_1 \to M \to S_2 \to 0$ be a non-split short exact sequence in $\mods A$. Like in the previous paragraph, it follows that $M$ arises as a label of an arrow in the half of the polygon $[0, \Filt_A\{ S_1, S_2\}] \subseteq \tors A$, which is contained in $\Phi_I^{-1}(\Filt_{A/I}\{S_2\})$. However, by \cref{eq:samebricks} it must also arise in $\Phi_I^{-1}(0)$. Since $M \not \in S_1^{\perp_A}$, $M$ cannot label an arrow above $\Filt_A\{S_1\}$ in $\Hasse(\tors A)$. Consequently, there must exist another simple module $S_3 \in \mods A$ such that $M \in \brick[\Filt_A\{S_3\}, \pi_{\uparrow} 0]$. As $\Filt_A\{S_3\} \in \Phi_I^{-1}(0)$ it follows that $S_3 \in \Scal_2$. In other words, there exists an arrow $\Tcal' \xleftarrow{M} \Tcal \subseteq \Hasse([\Filt_A\{ S_3\}, \pi_{\uparrow} 0])$ labelled by $M$. Because $M \in \Tcal$ it follows that $S_2 \in \Fac M \subseteq \Tcal$, and moreover that $S_2 \not \in (\Tcal')^{\perp_A}$, as otherwise there would be two bricks $M$ and $S_2$ in the intersection $(\Tcal')^{\perp_A} \cap \Tcal$, which is a contradiction by \cite[Thm. 3.3(b)]{DIRRT2017}. Any morphism to the simple module $S_2$ is an epimorphism, and since torsion classes are closed under quotients, it follows that $S_2 \in \Tcal'$. As a consequence the join $\Filt_A\{S_2\} \lor \Filt_A\{S_3\} = \Filt_A\{S_2, S_3\}$ is contained in $\Tcal'$. In particular, there exists an arrow $\Tcal'' \xleftarrow{S_2} \Filt_A\{S_2, S_3\}$ in $\Hasse(\tors A)$ for some $\Tcal''$ containing $\Filt_A\{S_3\}$. Combining these observations, it follows that
     \[ 0 \subseteq \Filt_A\{S_3\} \subseteq \Tcal'' \subseteq \Filt_A\{S_2,S_3\} \subseteq \Tcal' \subseteq \pi_{\uparrow} 0.\]
    In conclusion, the interval $[0, \pi_{\uparrow} 0]$ contains arrows labelled by $S_2$ which are not contracted by $\Phi_I$. Hence the equivalence class $[0]$ of $\Phi_I$ is not an interval, a contradiction to \cref{prop:quotlattiso}. Therefore $\Ext_A^1(S_2, S_1)=0$. \\

    In conclusion, we obtain $A \cong A/\langle \epsilon_1 \rangle \times A / \langle \epsilon_2 \rangle$, and clearly $\tors A / \langle \epsilon_2 \rangle \cong \tors A/I$ as required. 
\end{proof}

\section{Epimorphisms and lifting $\tau$-perpendicular intervals}\label{sec:epimorphisms}
In this section, let $A$ be $\tau$-tilting finite $K$-algebra, so that that $\tors A = \ftors A$ is finite. Let $I$ denote an ideal of $A$. Under this assumption it is possible to lift $\tau$-perpendicular intervals of $\tors A/I$ to $\tau$-perpendicular intervals of $\tors A$ in a more precise way than in the proof of \cref{prop:surjonobj}. This allows us to study the image of the functor $F_I$ better. Recall from  \cref{prop:quotlattiso} the isomorphism of lattices $\pi_{\uparrow}^{\Phi_I} (\tors A) \cong \tors A/I$. By a slight abuse of notation, given $\Tcal \in \tors A/I$ denote by $\pi_{\uparrow}^{\Phi_I} \Tcal$, or $\pi_{\uparrow} \Tcal$ for short, the element top element $\pi_{\uparrow}^{\Phi_I} \Phi_I^{-1}(\Tcal)$ of the preimage of $\Tcal$ under the lattice congruence $\Phi_I$. As a first step, consider the lemma which allows us to lift $\tau$-perpendicular intervals of $\tors A/I$ to $\tau$-perpendicular intervals of $\tors A$ explicitly.

\begin{lem} \label{prop:wideintinpreim}
   Let $I \in \ideal A$. For every $\tau$-perpendicular interval $[\Ucal, \Tcal] \subseteq \tors A/I$ there exists a $\tau$-perpendicular interval $[\Acal_{\Ucal}, \Bcal_{\Tcal}] \subseteq [\pi_{\uparrow} \Ucal, \pi_{\uparrow} \Tcal] \subseteq \tors A$. The interval is such that $\overline{[\Acal_{\Ucal}, \Bcal_{\Tcal}]} = [\Ucal, \Tcal]$ and satisfies $\Acal_{\Ucal}^{\perp_A} \cap \Bcal_{\Tcal} = \iota( \Ucal^{\perp_{A/I}} \cap \Tcal)$, where $\iota$ is the inclusion of \cref{eq:wideinclusion}. Moreover, it is given by
   \[ [\Acal_{\Ucal}, \Bcal_{\Tcal}] = [\pi_{\uparrow} \Ucal, \pi_{\uparrow} \Ucal \lor \Trm(\Wcal)], \]
   where $\Wcal = \iota(\Ucal^{\perp_{A/I}} \cap \Tcal)$.
\end{lem}

\begin{proof}
   Let $[\Ucal, \Tcal]$ be a $\tau$-perpendicular interval of $\tors A/I$ whose corresponding wide subcategory is given by $\Ucal^{\perp_{A/I}} \cap \Tcal = \Filt_{A/I}\{ S_1, \dots, S_k\} \in \wide A/I$ for some $1 \leq k \leq |A|$. By \cref{thm:Ringelbij}, it is generated by some semibrick $\{S_1, \dots, S_k\} \in \sbrick A$. This implies that 
   \begin{equation}\label{eq:sbrickcontained} \{S_1, \dots, S_k\} \subseteq \Ucal^{\perp_{A/I}} \cap \Tcal = (\pi_{\uparrow} \Ucal)^{\perp_{A}} \cap \pi_{\uparrow} \Tcal \cap \mods A/I \subseteq (\pi_{\uparrow} \Ucal)^{\perp_A} \cap \pi_{\uparrow} \Tcal \end{equation}
   by using \cref{cor:wideinquot} to obtain the equality. Since $(\pi_{\uparrow} \Ucal)^{\perp_A}$ and $\pi_{\uparrow} \Tcal$ are a torsion-free and a torsion class of $\mods A$, respectively, they are closed under extensions. As a consequence the lifted $\tau$-perpendicular subcategory $\Wcal = \Filt_A\{S_1, \dots, S_k\} \in \wide A$ satisfies $\Wcal \subseteq (\pi_{\uparrow} \Ucal)^{\perp_A} \cap \pi_{\uparrow} \Tcal$, as $\Wcal = \Filt_A\{S_1, \dots, S_k\}$ consists of iterated extensions of modules contained in $(\pi_{\uparrow} \Ucal)^{\perp_A} \cap \pi_{\uparrow} \Tcal$. \\

   By the dual of \cite[Prop. 9-5.10]{Reading2016} and \cref{prop:quotlattiso}, the intersection with $\mods A/I$ induces a bijection from the elements $\Vcal \in \tors A$ covering $\pi_{\uparrow} \Ucal$ in $\tors A$ to the elements $\Vcal'$ covering $\Ucal$ in $\tors A/I$. Thus, using \cref{prop:brickpreserve} and \cref{thm:brickcontract}, there is a brick-label preserving bijection between arrows
   \begin{equation}\label{eq:temporaryarrows} 
	\begin{aligned}
   &\left \{ \pi_{\uparrow} \Ucal \lor \Trm_A(S_i) \xrightarrow{S_i} \pi_{\uparrow} \Ucal \text{ in } \Hasse([\pi_{\uparrow} \Ucal, \pi_{\uparrow} \Tcal]) \subseteq \Hasse(\tors A) \right \}_{i=1}^k  \\
   &\longleftrightarrow \left \{ \Ucal \lor \Trm_{A/I}(S_i) \xrightarrow{S_i} \Ucal \text{ in } \Hasse([\Ucal, \Tcal]) \subseteq \Hasse(\tors A/I) \right \}_{i=1}^k,
   \end{aligned}
   \end{equation}
   which are labelled by the bricks $\{S_1, \dots, S_k\}$ generating $\Ucal^{\perp_A} \cap \Tcal$ by \cite[Thm. 4.16]{DIRRT2017}.\\

It is easy to see that $\Trm_A(\Wcal) = \Trm_A(S_1) \lor \dots \lor \Trm_A(S_k)$, so that \cref{prop:DIRRT420} gives the following $\tau$-perpendicular interval by taking the join of all atoms:
   \[ [\Acal_\Ucal, \Bcal_{\Tcal}] \coloneqq [\pi_{\uparrow} \Ucal, \pi_{\uparrow} \Ucal \lor \Trm(\Wcal)] \subseteq \tors A. \]
   By \cite[Thm. 4.16]{DIRRT2017} this intervals satisfies $\Acal^{\perp_A} \cap \Bcal_\Tcal = \Wcal$. Since both $\Wcal \subseteq \pi_{\uparrow} \Tcal$ and $\pi_{\uparrow} \Ucal \subseteq \pi_{\uparrow} \Tcal$, the interval is such that $\overline{[\Acal_\Ucal, \Bcal_{\Tcal}]} \subseteq [\Ucal, \Tcal] \subseteq \tors A/I$. By \cref{cor:wideinquot} we have
   \[\overline{\Acal}_\Ucal^{\perp_{A/I}} \cap \overline{\Bcal}_\Tcal = \Filt_A \{ S_1, \dots, S_k\} \cap \mods A/I = \Filt_{A/I} \{S_1, \dots, S_k\}. \]
   Hence $\overline{\Acal}_\Ucal^{\perp_{A/I}} \cap \overline{\Bcal}_\Tcal = \Ucal^{\perp_{A/I}} \cap \Tcal$ and by \cref{thm:wideintiso} there is a lattice isomorphism between $\overline{[\Acal_{\Ucal}, \Bcal_{\Tcal}]}$ and $[\Ucal, \Tcal]$. Since $\overline{[\Acal_{\Ucal}, \Bcal_{\Tcal}]} \subseteq [\Ucal, \Tcal]$ it follows that $\overline{[\Acal_{\Ucal}, \Bcal_{\Tcal}]} = [\Ucal, \Tcal]$. 
\end{proof}

Using \cref{prop:wideintinpreim}, define a map of $\tau$-perpendicular intervals:
\begin{equation}\label{eq:liftingintervals}
    \begin{aligned}
        \ifrak: \tauint(\tors A/I) &\to \tauint(\tors A)\\
        [\Ucal, \Tcal] &\mapsto \ifrak [\Ucal, \Tcal] = [\pi_{\uparrow} \Ucal, \pi_{\uparrow} \Ucal \lor \Trm(\Wcal)].
    \end{aligned}
\end{equation}

\begin{exmp}\label{rmk:noninclusion}
    The map of $\tau$-perpendicular intervals $\ifrak: \tauint(\tors A/I) \to \tauint(A)$ of \cref{eq:liftingintervals} is not inclusion-preserving. For example, take the surjective algebra morphism $A \cong K(1 \to 2) \twoheadrightarrow K^2$. Then the inclusion $[\Fac(\begin{smallmatrix} 1 \end{smallmatrix}), \Fac (\begin{smallmatrix} 1 \end{smallmatrix})] \subseteq [0, \Fac( \begin{smallmatrix} 1 \end{smallmatrix})]$ of $\tau$-perpendicular intervals of $\tors K^2$ maps to $[\Fac( \begin{smallmatrix} 1\\2 \end{smallmatrix}),\Fac( \begin{smallmatrix} 1\\2 \end{smallmatrix})] \not \subseteq [0, \Fac (\begin{smallmatrix}1 \end{smallmatrix})]$.
\end{exmp}

Nonetheless, there is a way of resolving this problem by restricting the lattice congruence to the desired interval.

\begin{prop}\label{prop:essentialimage}
    For every inclusion of $\tau$-perpendicular intervals $[\Ucal, \Tcal] \subseteq [\Vcal, \Scal]$ in $\tors A/I$, there exists a $\tau$-perpendicular interval $\overline{\ifrak}_{\Vcal}^{\Scal} [\Ucal, \Tcal] \subseteq \ifrak[\Vcal, \Scal]$ such that $(\overline{\ifrak}_{\Vcal}^{\Scal} [\Ucal, \Tcal]) \cap \mods A/I= [\Ucal, \Tcal]$ and its corresponding wide subcategory is $\iota(\Ucal^{\perp_{A/I}} \cap \Tcal)$. The intervals are such that  
    \[ F_I([f_{(\ifrak[\Vcal, \Scal])(\overline{\ifrak}_{\Vcal}^{\Scal} [\Ucal, \Tcal])}]) = [f_{[\Vcal, \Scal][\Ucal, \Tcal]}],\]
    consequently every morphism of $\Tfrak(A/I)$ lies in the essential image of $F_I: \Tfrak(A) \to \Tfrak(A/I)$. 
\end{prop}
\begin{proof}
    Let $\Ucal, \Vcal, \Scal, \Tcal \in \tors A/I$ be such that $[\Ucal, \Tcal] \subseteq [\Vcal, \Scal]$ and both are $\tau$-perpendicular intervals of $\tors A/I$. By definition of $\Tfrak(A/I)$ every non-zero morphism in $\Tfrak(A/I)$ is of the form $[f_{[\Vcal, \Scal][\Ucal, \Tcal]}] \in \Tfrak(A/I)$. Consider the representative $f_{[\Vcal, \Scal][\Ucal,\Tcal]}$ of its equivalence class. Using the map $\ifrak$ of \cref{eq:liftingintervals} we obtain a $\tau$-perpendicular interval $\ifrak[\Vcal, \Scal]$ of $\tors A$ such that $F_I((\ifrak[\Vcal, \Scal])_\sim)=[\Vcal, \Scal]_\sim$. However, the $\tau$-perpendicular interval $\ifrak[\Ucal, \Tcal]$ may not be contained in $\ifrak[\Vcal, \Scal]$, see \cref{rmk:noninclusion}. \\

    Nonetheless, we may consider the restriction of the lattice congruence $\Phi \coloneqq \Phi_I$ to the interval lattice $\ifrak[\Vcal, \Scal]$, which we denote by $\overline{\Phi}|_{\Vcal}^{\Scal}$. By the dual of \cite[Lem. 9-5.7]{Reading2016}, the interval $\overline{[\Vcal, \Scal]} = [\Vcal, \Scal]$ of $\tors A/I$, which contains $[\Ucal, \Tcal]$, is isomorphic to the quotient lattice $\ifrak[\Vcal, \Scal]/(\Phi|_{\Vcal}^{\Scal})$. Let $\Wcal \coloneqq \iota(\Ucal^{\perp_{A/I}} \cap \Tcal)$ and let $\{S_1, \dots, S_k\}$ be the relative simple modules of $\Ucal^{\perp_{A/I}} \cap \Tcal$, so that $\Wcal= \Filt_A\{S_1, \dots, S_k\}$ by \cref{thm:Ringelbij}.\\
    
    Restricting \cref{prop:quotlattiso} to $[\Vcal, \Scal]$, we find that the preimage of $\Ucal$ under $\overline{\Phi}|_{\Vcal}^{\Scal}$ is an interval of $[\Vcal, \Scal]$ which we denote by $[(\overline{\pi}|_{\Vcal}^{\Scal})_{\downarrow} (\Ucal), (\overline{\pi}|_{\Vcal}^{\Scal})_{\uparrow} (\Ucal)] \subseteq \ifrak[\Vcal, \Scal]$. Let $\Ucal' = (\overline{\pi}|_{\Vcal}^{\Scal})_{\uparrow} (\Ucal) \in \tors A$. Similar to the proof of \cref{prop:wideintinpreim}, we apply the dual of \cite[Prop. 9-5.10]{Reading2016} and \cref{prop:brickpreserve} to obtain a brick-label preserving bijection
    \begin{equation}\label{eq:temporaryarrows2} 
	\begin{aligned}
       &\left \{ \Ucal' \lor \Trm_A(S_i) \xrightarrow{S_i}  \Ucal' \text{ in } \Hasse([(\overline{\pi}|_{\Vcal}^{\Scal})_{\uparrow} \Ucal, (\overline{\pi}|_{\Vcal}^{\Scal})_{\uparrow} \Tcal ]) \subseteq \Hasse(\ifrak[\Vcal, \Scal]) \right \}_{i=1}^k  \\
       &\longleftrightarrow \left \{ \Ucal \lor \Trm_{A/I}(S_i) \xrightarrow{S_i} \Ucal \text{ in } \Hasse([\Ucal, \Tcal]) \subseteq \Hasse([\Vcal, \Scal]) \right \}_{i=1}^k,
       \end{aligned}
       \end{equation}
       which are labelled by the bricks $\{S_1, \dots, S_k\}$ generating $\Ucal^{\perp_{A/I}} \cap \Tcal$ by \cite[Thm. 4.16]{DIRRT2017}. Since $\Tfrak(\Wcal_1) = \bigvee_{i=1}^k \Tfrak(S_i)$ it follows from \cref{prop:DIRRT420} that
       \begin{equation}\label{eq:2nditvlift}
           \overline{\ifrak}_\Vcal^{\Scal}[\Ucal, \Tcal] \coloneqq [(\overline{\pi}_{\Vcal}^{\Scal})_\uparrow (\Ucal), (\overline{\pi}_{\Vcal}^{\Scal})_\uparrow (\Ucal) \lor \Trm(\Wcal)]
       \end{equation} 
       is a $\tau$-perpendicular interval of $\tors A$ such that the corresponding wide subcategory is $\Wcal$ by \cite[Thm. 4.16]{DIRRT2017}. Write $\Tcal' = (\overline{\phi}_{\Vcal}^{\Scal})_\uparrow \Tcal$. Because $(\overline{\Phi}|_{\Vcal}^{\Scal})^{-1}(\Ucal)\subseteq \Phi^{-1}(\Ucal)$ and $(\overline{\Phi}|_{\Vcal}^{\Scal})^{-1}(\Tcal)\subseteq \Phi^{-1}(\Tcal)$, we have that $\overline{\Ucal'} = \Ucal$ and $\overline{\Tcal'} = \Tcal$. It follows from  \cref{cor:wideinquot} that therefore $(\overline{\Ucal'})^{\perp_A} \cap \Tcal' \cap \mods A/I = \Ucal^{\perp_{A/I}} \cap \Tcal$. Now the argument following \cref{eq:sbrickcontained} in the proof of \cref{prop:wideintinpreim} applies to give $\Wcal \subseteq \Tcal'$. Since furthermore $\Ucal' \subseteq \Tcal'$, it follows that $(\overline{\ifrak}_{\Vcal}^{\Scal} [\Ucal, \Tcal]) \cap \mods A/I \subseteq [\Ucal, \Tcal]$. \\
       
       Again, analogous to the proof of \cref{prop:wideintinpreim}, it follows from \cref{thm:wideintiso} and \cref{cor:wideinquot} that actually $(\overline{\ifrak}_{\Vcal}^{\Scal} [\Ucal, \Tcal]) \cap \mods A/I = [\Ucal, \Tcal]$ because there exists a lattice isomorphism between the two. To complete the proof we show that $\Ucal' \lor \Trm(\Wcal)  \subseteq \pi_{\uparrow} \Vcal \lor \Trm(\iota(\Vcal^{\perp_{A/I}} \cap \Scal))$. Indeed,
       \[ \Ucal' \lor \Trm(\Wcal) \subseteq \Ucal' \lor \Trm(\iota(\Vcal^{\perp_{A/I}} \cap \Scal)) \subseteq \pi_{\uparrow} \Vcal \lor \Trm(\iota(\Vcal^{\perp_{A/I}} \cap \Scal)),\]
       where the first inequality follows since $\Ucal^{\perp_{A/I}} \cap \Tcal \subseteq \Vcal^{\perp_{A/I}} \cap \Scal$ lifts to $\mods A$ and the second inequality follows from the same observation and the fact that $\Ucal' \subseteq \pi_\uparrow \Vcal \lor \Trm(\iota(\Vcal^{\perp_{A/I}} \cap \Scal))$ by construction. In conclusion, given an inclusion of $\tau$-perpendicular intervals $[\Ucal, \Tcal] \subseteq [\Vcal, \Scal]$ in $\tors A/I$, there exists an inclusion of $\tau$-perpendicular intervals
       \[ (\overline{\ifrak}_{\Vcal}^{\Scal}[\Ucal, \Tcal]) \subseteq \ifrak[\Vcal, \Scal]\]
       such that $(\overline{\ifrak}_{\Vcal}^{\Scal}[\Ucal, \Tcal]) \cap \mods A/I = [\Ucal, \Tcal]$ and $\ifrak[\Vcal, \Scal]\cap \mods A/I = [\Vcal, \Scal]$. Consequently, $F_I((\overline{\ifrak}_{\Vcal}^{\Scal}[\Ucal, \Tcal])_\sim) = [\Ucal, \Tcal]_\sim$ and $F_I((\ifrak[\Vcal, \Scal])_\sim) = [\Vcal, \Scal]_\sim$. In conclusion any morphism $[f_{[\Vcal, \Scal][\Ucal, \Tcal]}] \in \Tfrak(A/I)$ may be obtained by applying $F_I$ to $[f_{(\ifrak[\Vcal, \Scal])(\overline{\ifrak}_{\Vcal}^{\Scal}[\Ucal, \Tcal])}].$
\end{proof}

Recall that an epimorphism $e$ in a category $\Ccal$ is called \textit{extremal} if whenever one can write $e= m \circ f$, with $m$ a monomorphism, then $m$ is an isomorphism. Let $\mathcal{C}\mathrm{at}$ denote the category of small categories.

\begin{cor}\label{cor:extremalepi}
    The smallest subcategory of $\Tfrak(A/I)$ containing the image of $F_I$ is $\Tfrak(A/I)$ itself. Thus $F_I$ is an extremal epimorphism in the category $\mathcal{C}\mathrm{at}$.
\end{cor}
\begin{proof}
    It follows immediately from \cref{prop:essentialimage} that the image of $F_I$ is $\Tfrak(A/I)$ itself. The functor $F_I$ of \cref{thm:inducedfunctor} is an extremal epimorphism in $\mathcal{C}$at by \cite[Thm. 3.4]{BBP1999}.
\end{proof}

Moreover, we say that a functor $G: \Acal \to \Bcal$ \textit{reflects composition} if given two morphisms $f$ and $g$ in $\Acal$ such that $G(f) \circ G(g)$ is defined in $\Bcal$, there exist morphisms $f'$ and $g'$ in $\Acal$ such that $F(f) = G(f')$, $G(g) = G(g')$ and $f' \circ g'$ is defined in $\Bcal$.

\begin{lem}\label{lem:reflectcomp}
    The functor $F_I: \Tfrak(A) \to \Tfrak(A/I)$ reflects composition. Moreover, if $[f_{[\Ucal, \Tcal][\Xcal, \Ycal]}] \circ [f_{[\Vcal, \Scal][\Ucal', \Tcal']}]$ is defined in $\Tfrak(A/I)$ then the composition 
    \[ [f_{(\ifrak[\Ucal, \Tcal])(\overline{\ifrak}_{\Ucal}^{\Tcal}[\Xcal, \Ycal])}] \circ [f_{(\ifrak[\Vcal, \Scal])(\overline{\ifrak}_{\Vcal}^{\Scal}[\Ucal', \Tcal'])}] \]
    is defined in $\Tfrak(A)$, where the maps of $\tau$-perpendicular intervals are as in \cref{eq:liftingintervals} and \cref{eq:2nditvlift}.
\end{lem}
\begin{proof}
    By \cref{lem:composition1} we may assume $\Ucal = \Ucal'$ and $\Tcal = \Tcal'$. From \cref{prop:wideintinpreim} we know that the interval $\ifrak[\Ucal,\Tcal] \subseteq \tors A$ corresponds to the wide subcategory $\iota(\Ucal^{\perp_{A/I)}} \cap \Tcal) \in \wide A$. On the other hand, from \cref{prop:essentialimage} we know that the interval $\overline{\ifrak}_{\Vcal}^{\Scal}[\Ucal, \Tcal]$ corresponds to the wide subcategory $\iota(\Ucal^{\perp_{A/I)}} \cap \Tcal)$ as well. Therefore $(\ifrak_{\Vcal}^{\Scal}[\Ucal, \Tcal])_\sim = (\ifrak[\Ucal, \Tcal])_\sim$ in $\Tfrak(A)$. It follows from \cref{lem:composition1} that the composition is defined in $\Tfrak(A)$. It is clear from \cref{prop:essentialimage} that the functor $F_I$ sends this composition in $\Tfrak(A)$ to the desired one in $\Tfrak(A/I)$.
\end{proof}

\begin{cor}\label{cor:regularepi}
    The functor $F_I: \Tfrak(A) \to \Tfrak(A/I)$ is a regular epimorphism in $\mathcal{C}$at, that is, it is the coequaliser of a pair of morphisms in $\mathcal{C}\mathrm{at}$.
\end{cor}
\begin{proof}
    The result \cite[Prop. 5.1]{BBP1999} states that an extremal epimorphism in $\mathcal{C}\mathrm{at}$ which reflects composition is regular. Therefore, the result follows from  \cref{cor:extremalepi} and \cref{lem:reflectcomp}.
\end{proof}

\begin{exmp}
    Continuing with the surjective algebra morphisms of \cref{rmk:noninclusion}, we can see that the assignment from morphisms of $\Tfrak(A/I)$ to morphisms of $\Tfrak(A)$ given by
    \[ [f_{[\Vcal, \Scal][\Ucal, \Tcal]}] \mapsto [f_{(\ifrak[\Vcal, \Scal])(\overline{\ifrak}_{\Vcal}^{\Scal}[\Ucal, \Tcal])}]\]
    is not well-defined on composition. Take for example the morphism $[f_{[0, \mods A][\Fac( \begin{smallmatrix} 1 \end{smallmatrix}), \Fac (\begin{smallmatrix} 1 \end{smallmatrix}) ]}]$ of $\Tfrak(A/I)$ which gets mapped to $[f_{[0, \mods A][\Fac (\begin{smallmatrix} 1\\2 \end{smallmatrix}), \Fac (\begin{smallmatrix} 1\\2 \end{smallmatrix}) ]}] \in \Tfrak(A)$ under the assignment. However, the morphism may be decomposed as
    \begin{align*} 
    [f_{[0, \Fac (\begin{smallmatrix} 1 \end{smallmatrix})][\Fac(\begin{smallmatrix} 1 \end{smallmatrix}), \Fac (\begin{smallmatrix} 1 \end{smallmatrix})]}]\circ  [f_{[0, \mods A][0, \Fac(\begin{smallmatrix} 1 \end{smallmatrix})]}] &= [f_{[0, \mods A][\Fac (\begin{smallmatrix} 1 \end{smallmatrix}), \Fac(\begin{smallmatrix} 1 \end{smallmatrix})]}] \\
    &= [f_{[\Fac(\begin{smallmatrix} 1 \end{smallmatrix}), \mods A][\Fac (\begin{smallmatrix} 1 \end{smallmatrix}), \Fac (\begin{smallmatrix} 1 \end{smallmatrix})]}] \circ [f_{[0, \mods A][\Fac (\begin{smallmatrix} 1 \end{smallmatrix}), \mods A]}]
    \end{align*}
    However, the composition of the images of the left-hand side is $[f_{[0, \mods A][\Fac(\begin{smallmatrix} 1 \end{smallmatrix}),\Fac (\begin{smallmatrix} 1 \end{smallmatrix})]}]$, which is different to the composition of the images of the left-hand side $[f_{[0, \mods A][\Fac (\begin{smallmatrix} 1\\2 \end{smallmatrix} ), \Fac (\begin{smallmatrix} 1\\2 \end{smallmatrix})]}]$.
\end{exmp}

\section{Classifying spaces and picture groups} \label{sec:classifyingspace}

The original motivation for the ($\tau$-)cluster morphism category comes from the study of cohomology groups of the \textit{picture group} \cite{IgusaTodorovWeyman2016}. This group is the fundamental group of an associated topological space, called the \textit{picture space} which is simply the classifying space $\Bcal \Tfrak(A)$ of the ($\tau$)-cluster morphism category $\Tfrak(A)$, see \cite{HansonIgusa2021,IgusaTodorov2017}. Therefore, if $\Bcal \Tfrak(A)$ is a $K(\pi,1)$ space, that is, a space whose only non-trivial homotopy group is in degree 1, then we may relate the cohomology groups of the picture group with those of the space. \\

The \textit{classifying space} $\Bcal \Ccal$ of a topological category $\Ccal$ was first explicitly studied in \cite{Segal1968}, where it is defined as the geometric realisation of the nerve of the category. The nerve of the category is a simplicial set whose 0-simplices are identity morphisms and whose $k$-simplices are chains of $k$ composable non-identity morphisms. It was shown in \cite{IgusaTodorov2017} and \cite{HansonIgusa2021} that for the ($\tau$)-cluster morphism category $\Bcal \Tfrak(A)$ is actually a CW-complex with a particularly nice structure. As a consequence we are able to obtain the following.

\begin{prop} \label{thm:spacequotient}
    Let $A$ be $\tau$-tilting finite and $I \in \ideal A$, then the classifying space $\Bcal \Tfrak(A/I)$ is a quotient space of $\Bcal \Tfrak(A)$ and the quotient map is induced by $F_I$ from \cref{thm:inducedfunctor}.
\end{prop}
\begin{proof}
    The $\tau$-cluster morphism category $\Tfrak(A)$ is a cubical category \cite[Thm. 2.14]{HansonIgusa2021}, see also \cite[Thm. 3.16]{Kaipelcatpartfan}. This means that every morphism can be seen as (the diagonal of) a cube, whose edges correspond to factorisations of the morphism into irreducible ones. It is shown in \cite[Prop. 4.7]{HansonIgusa2021} that the classifying space $\Bcal \Tfrak(A)$ of the $\tau$-cluster morphism category is a CW-complex with one $k$-cell $e([\Ucal, \Tcal]_\sim)$ for each equivalence class of wide intervals, where $k$ is the number of isomorphism classes of (relative) simple modules of $\Ucal^\perp \cap \Tcal$. This cell is the union of factorisation cubes of morphisms $[f_{[\Ucal, \Tcal] [\Xcal, \Xcal]}]$ for $\Xcal \in [\Ucal, \Tcal]$. Given an ideal $I \in \ideal A$, define an equivalence $\asymp$ relation on factorisation cubes by setting
    \[ [f_{[\Ucal, \Tcal][\Xcal, \Xcal]]}] \asymp [f_{[\Ucal', \Tcal'][\Xcal', \Xcal']}]\]
    whenever $F_I$ applied to these morphisms coincides. Since factorisation cubes are simply geometric realisations of morphisms, this identification coincides with the generalised congruence $\asymp_I$, as defined in \cite[Sec. 3]{BBP1999}, on $\Tfrak(A)$ induced by the functor $F_I$. By \cite[Cor. 3.11]{BBP1999} there is a monomorphism $(\Tfrak(A)/\asymp ) \to \Tfrak(A/I))$ in $\mathcal{C}\mathrm{at}$ from the quotient category $\Tfrak(A)/\asymp$ as defined in \cite[Sec. 3.9]{BBP1999}. Since $F$ is an extremal epimorphism by \cref{cor:extremalepi}, it follows that $(\Tfrak(A)/\asymp) \cong \Tfrak(A/I)$. Clearly, $\Bcal(\Tfrak(A)/\asymp)$ is a quotient space of $\Bcal \Tfrak(A)$ and the result follows.
\end{proof}

The standard approach taken in \cite{IgusaTodorov2017,IgusaTodorov22, HansonIgusaPW2SMC,HansonIgusa2021,Kaipelcatpartfan} to show that $\Bcal \Tfrak(A)$ is a $K(\pi,1)$ space is to use three sufficient conditions translated from cube complexes \cite{Gromov1987} to cubical categories \cite{Igusa2022}. The first condition is the existence of faithful functor to a groupoid with one object (a so-called \textit{faithful group functor}). While the existence of such a faithful group functor is conjectured to exist for an arbitrary algebra, see \cite[Conj. 5.10]{HansonIgusa2021}, its existence is currently known only for the following three classes of algebras:
\begin{itemize}
    \item Hereditary algebras of finite or tame type \cite[Thm. 3.7]{IgusaTodorov22};
    \item $K$-stone algebras, \cite[Thm. 5.9]{HansonIgusaPW2SMC};
    \item Algebras whose $g$-vector fan is a finite hyperplane arrangement \cite[Thm. 6.20]{Kaipelcatpartfan}.
\end{itemize}

As an immediate consequence of \cref{cor:equivcats} we can extend these faithful functors to $\tau$-tilting finite algebras $A$ having isomorphic lattice of torsion classes to one of the above families.\\

In general, the picture group captures the structure of the lattice of torsion-classes and thus of maximal green sequences. For hereditary algebras, it was shown in \cite{IgusaTodorovMGS2021} that positive expressions for a certain (Coxeter) group element are in bijection with maximal green sequences. The definition of the picture group in \cite[Prop. 4.4]{HansonIgusa2021} is readily adapted to the general case as follows.

\begin{defn}
    Let $A$ be a finite-dimensional algebra. The \textit{picture group} $G(A)$ is defined by having generators 
    \[\{ X_S: S \in \brick A \text{ and } \Filt_A\{S\} \text{ is $\tau$-perpendicular} \} \cup \{ g_{\Tcal} : \Tcal \in \ftors A \}\]
    with a relation $g_{\Tcal_1}= X_S g_{\Tcal_2}$ whenever there is an arrow $\Tcal_1 \xrightarrow{S} \Tcal_2$ in $\Hasse(\ftors A)$ and the relation $g_0 = e$. 
\end{defn}

We conclude this section with a group-theoretic analogue of \cref{thm:spacequotient}

\begin{prop} \label{prop:picgroupquot}
    Let $A$ be $\tau$-tilting finite and $I \in \ideal A$. Then there is a surjective group homomorphism $G(A) \to G(A/I)$ induced by $\overline{(-)}: \tors A \to \tors A/I$. 
\end{prop}
\begin{proof}
    Define a map $\phi: G(A) \to G(A/I)$ given by
    \[ X_S \mapsto \begin{cases} X_S & \text{ if } S \in \brick A/I, \\ e &\text{ otherwise,}\end{cases} \quad \text{and} \quad g_\Tcal \mapsto g_{\overline{\Tcal}}.\]
    Since $\mods A/I$ is a full subcategory of $\mods A$, we identify $\brick A/I$ with $\{ S \in \brick A: IS = 0\} \subseteq \brick A$, similar to \cite[Sec. 5.2]{DIRRT2017}. Therefore $\phi$ induces a surjection on the generators $X_S \in G(A/I)$. Similarly by \cref{thm:DIRRTquot} the map $\overline{(-)}: \tors A \to \tors A/I$ is surjective and hence $\phi$ is surjective on the generators $g_{\overline{\Tcal}}$. To show that the group relations $g_{\Tcal_1} = X_S g_{\Tcal_2}$ are preserved we distinguish between two cases: If $S \in \brick A/I$, the corresponding arrow of $\Hasse(\tors A)$ is not contracted, by \cref{thm:brickcontract}, and the group relation becomes 
    \[ \phi(g_{\Tcal_1}) = g_{\overline{\Tcal}_1} = X_S g_{\overline{\Tcal}_2} = \phi(X_S) \phi(g_{\Tcal_2}).\]
    And if $S \not \in \brick A/I$, then $\phi(X_S)=e$ and the corresponding arrow of $\Hasse(\tors A)$ is contracted, so $\overline{\Tcal_1} = \overline{\Tcal_2}$ and the group relation becomes
    \[ \phi(g_{\Tcal_1}) = g_{\overline{\Tcal}_1} = g_{\overline{\Tcal}_2} = e g_{\overline{\Tcal_2}} = \phi(X_S) \phi(g_{\Tcal_2}).\]
    Hence $\phi$ is a well-defined group homomorphism and surjective.
\end{proof}

\section{Examples} \label{sec:examples}
We begin with an example illustrating the different presentations $\Wfrak(A), \Cfrak(A)$ and $\Tfrak(A)$ of the $\tau$-cluster morphism category. 
\begin{exmp}\label{exmp:good}
Let $A \coloneqq K(1 \xrightarrow{a} 2)$ and consider $K^2 \cong A/\langle a \rangle \cong K(1 \quad 2)$. We write modules using their socle series and display the objects of the $\tau$-cluster morphism categories in \cref{tab:table1}. 
Recall that the correspondence between the objects of the $\tau$-cluster morphism categories is given as follows:
\begin{enumerate}
    \item $\Cfrak(A) \ni [\Ccal_{(M,P)}] \mapsto M^\perp \cap {}^\perp \tau M \cap P^\perp \in \Wfrak(A)$.
    \item $\Tfrak(A) \ni [\Ucal, \Tcal]_\sim \mapsto \Ucal^\perp \cap \Tcal \in \Wfrak(A)$.
    \item $\Cfrak(A) \ni [\Ccal_{(M,P)}] \mapsto [\Fac M, {}^\perp \tau M \cap P^\perp]_\sim \in \Tfrak(A)$. 
\end{enumerate}

\begin{table}[ht!]
\begin{center}
\bgroup
\def\arraystretch{1.5}
\begin{tabular}{||c c c ||} 
 \hline
 $\Wfrak(A)$ & $\Cfrak(A)$  & $\Tfrak(A)$  \\ [0.5ex] 
 \hline\hline
 $\mods A$ & $[\Ccal_{(0,0)}]$ & {\small$[0, \mods A]_{\sim}$} \\ 
 \hline
 $\Filt\{1\}$ & $[\Ccal_{(\begin{smallmatrix}2\end{smallmatrix},\begin{smallmatrix} 0 \end{smallmatrix})}] = [\Ccal_{(\begin{smallmatrix} 0 \end{smallmatrix},\begin{smallmatrix}2\end{smallmatrix})}]$ & {\small$[0, \Fac (\begin{smallmatrix}1\end{smallmatrix})]_\sim = [\Fac (\begin{smallmatrix}2\end{smallmatrix}), \mods A]_\sim$} \\
 \hline
 $\Filt\{ \begin{smallmatrix} 2 \end{smallmatrix}\}$& $[\Ccal_{(\begin{smallmatrix} 1\\2 \end{smallmatrix},\begin{smallmatrix} 0 \end{smallmatrix})}] = [\Ccal_{(\begin{smallmatrix} 0 \end{smallmatrix},\begin{smallmatrix} 1\\2 \end{smallmatrix})}]$ & {$[\begin{smallmatrix} 0 \end{smallmatrix}, \Fac (\begin{smallmatrix}2 \end{smallmatrix})]_\sim = [\Fac (\begin{smallmatrix} 1\\2 \end{smallmatrix}), \mods A]_\sim$} \\
 \hline
  $\Filt\{\begin{smallmatrix} 1\\2 \end{smallmatrix} \}$& $[\Ccal_{(\begin{smallmatrix} 1\end{smallmatrix},\begin{smallmatrix} 0 \end{smallmatrix})}]$ & {\small$[\Fac (\begin{smallmatrix} 1 \end{smallmatrix}), \Fac (\begin{smallmatrix} 1\\2 \end{smallmatrix})]_\sim$} \\
 \hline
 0  & \begin{tabular}{@{}c@{}} {\small$[\Ccal_{(M,P)}]$} \\ {\small$(M,P) \in \ttilt A$} \end{tabular} & \begin{tabular}{@{}c@{}}{\small $[\Tcal,\Tcal]_{\sim}$ }\\ {\small$\Tcal \in \tors A$} \end{tabular}  \\ [1ex] 
 \hline
\end{tabular}
\egroup
\caption{Objects of the $\tau$-cluster morphism categories of $A$}
\label{tab:table1}
\end{center}
\end{table}

We use this translation to describe the image of the functor $F_I: \Tfrak(A) \to \Tfrak(A/I)$ in $\Wfrak(A)$ and $\Cfrak(A)$ in \cref{tab:table2}. Each entry in \cref{tab:table2} represents the image under  $F_{\langle a \rangle}: A \to K^2$  of the entry in the same position of \cref{tab:table1}. Using this description, we illustrate the image of one particular morphism under $F_I$ in \cref{tab:table3}. This example suggests that describing the functor $F_I: \Tfrak(A) \to \Tfrak(A/I)$ is less natural in $\Wfrak(A)$ and $\Cfrak(A)$. 

\begin{table}[ht!]
\begin{center}
\bgroup
\def\arraystretch{1.5}
\begin{tabular}{||c c c ||} 
 \hline
 $\Wfrak(K^2)$ & $\Cfrak(K^2)$  & $\Tfrak(K^2)$  \\ [0.5ex] 
 \hline\hline
 $\mods K^2$ & $[\Ccal_{(0,0)}]$ & {\small$[0, \mods K^2]_{\sim}$} \\ 
 \hline
 $\Filt\{ \begin{smallmatrix} 1 \end{smallmatrix}\}$ & $[\Ccal_{(2,0)}] = [\Ccal_{(0,2)}]$ & {\small$[0, \Fac \begin{smallmatrix} 1 \end{smallmatrix}]_\sim = [\Fac \begin{smallmatrix} 2 \end{smallmatrix}, \mods K^2]_\sim$} \\
 \hline
 $\Filt\{ \begin{smallmatrix} 2 \end{smallmatrix}\}$& $[\Ccal_{(1,0)}] = [\Ccal_{(0,1)}]$ & {\small$[0, \Fac \begin{smallmatrix} 2 \end{smallmatrix}]_\sim = [\Fac \begin{smallmatrix} 1 \end{smallmatrix}, \mods K^2]_\sim$} \\
 \hline
  $0$& $[\Ccal_{(\begin{smallmatrix} 1 \end{smallmatrix},\begin{smallmatrix} 2 \end{smallmatrix})}]$ & {\small$[\Fac \begin{smallmatrix} 1 \end{smallmatrix}, \Fac \begin{smallmatrix} 1 \end{smallmatrix}]_\sim$} \\
 \hline
 0  & \begin{tabular}{@{}c@{}} {\small$[\Ccal_{(M,P)}]$} \\ {\small$(M,P) \in \ttilt K^2$} \end{tabular} & \begin{tabular}{@{}c@{}}{\small $[\Tcal,\Tcal]_{\sim}$ }\\ {\small$\Tcal \in \tors K^2$} \end{tabular}  \\ [1ex] 
 \hline
\end{tabular}
\egroup
\caption{The corresponding images under $F_I$ of objects in \cref{tab:table1}.}
\label{tab:table2}
\end{center}
\end{table}

\begin{table}[ht!]
\begin{center}
\bgroup
\def\arraystretch{1.5}
\begin{tabular}{|| c c c ||} 
 \hline
 category / algebra & $A$ & $K^2$  \\ [0.5ex] 
 \hline\hline
 $\Wfrak(-)$ & $[(\begin{smallmatrix} 1 \end{smallmatrix},\begin{smallmatrix} 0 \end{smallmatrix})]: \mods A \to \Filt\{\begin{smallmatrix} 1\\2 \end{smallmatrix}\}$ & $[(\begin{smallmatrix} 1 \end{smallmatrix},\begin{smallmatrix} 2 \end{smallmatrix})]: \mods K^2 \to 0$\\ 
 \hline
 $\Cfrak(-)$ & $[f_{\Ccal_{(0,0)}\Ccal_{(1,0)}}]$ & $[f_{\Ccal_{(0,0)}\Ccal_{(1,2)}}]$ \\
 \hline
 $\Tfrak(-)$ & $[f_{[0, \mods A][\Fac (\begin{smallmatrix} 1 \end{smallmatrix}), \Fac (\begin{smallmatrix} 1\\2 \end{smallmatrix})]}$ & $[f_{[0, \mods K^2][\Fac (\begin{smallmatrix} 1 \end{smallmatrix}), \Fac (\begin{smallmatrix} 1 \end{smallmatrix})]}$ \\
 \hline
\end{tabular}
\egroup
\caption{The image of a morphism under $F_I$ in the different presentations of the $\tau$-cluster morphism category}
\label{tab:table3}
\end{center}
\end{table}
\end{exmp}

Next we show that the assumption of $\tau$-tilting finiteness in \cref{prop:surjonobj} is necessary by constructing an example where there exists a $\tau$-perpendicular interval $[\Ucal, \Tcal] \subseteq \tors A/I$ not in the image of any $\tau$-perpendicular interval of $\tors A$ under $\overline{(-)}: \tors A \to \tors A/I$. This shows that $F_I$ is not surjective-on-objects.

\begin{exmp}\label{exmp:taufinite}
    Let $A \coloneqq K( \begin{tikzcd} 1 \arrow[r, shift left, "b"] \arrow[r, shift right] & 2 \end{tikzcd})$ and $B \coloneqq A/\langle b \rangle \cong K( \begin{tikzcd} 1 \arrow[r] & 2 \end{tikzcd})$ then $A$ is well-known to be $\tau$-tilting infinite and $B$ to be $\tau$-tilting finite. Consider the interval $[ \Fac_{B} (\begin{smallmatrix}1\end{smallmatrix}), \Fac_{B} (\begin{smallmatrix}1\\2\end{smallmatrix})] \subseteq \tors B$, which is $\tau$-perpendicular and gives rise to the $\tau$-perpendicular wide subcategory $\Filt_{B}\{ \begin{smallmatrix}1\\2\end{smallmatrix} \} \subseteq \mods B$. \\

    The preimage $[\pi_{\downarrow} \Fac_{B} (\begin{smallmatrix}1\end{smallmatrix}), \pi_{\uparrow} \Fac_{B} (\begin{smallmatrix}1\\2\end{smallmatrix})] \subseteq \tors A$ of this interval under the surjection $\overline{(-)}: \tors A \to \tors B$ is $[\Fac_A (1), \Fac_A (\begin{smallmatrix}1\\22\end{smallmatrix})] \subseteq \tors A$ which is not $\tau$-perpendicular. Moreover the $\tau$-perpendicular intervals contained in this preimage come in four families of the forms
    \begin{align*} [\Fac_A(\tau_A^{-m} (\begin{smallmatrix}1\\22\end{smallmatrix})), \Fac_A(\tau_A^{-(m+1)}(\begin{smallmatrix}2\end{smallmatrix}))], \quad & [\Fac_A(\tau_A^{-(m+1)} (\begin{smallmatrix}2\end{smallmatrix})), \Fac_A(\tau_A^{-(m+1)}(\begin{smallmatrix}1\\22\end{smallmatrix}))], \\
    [\Fac_A(\tau_A^m (\begin{smallmatrix}1\end{smallmatrix})), \Fac_A(\tau_A^m (\begin{smallmatrix}11\\2\end{smallmatrix}))], \quad& [\Fac_A(\tau_A^m (\begin{smallmatrix}11\\2\end{smallmatrix}))], \Fac_A(\tau_A^{m+1} (\begin{smallmatrix}2\end{smallmatrix}))]
    \end{align*}
    for $m \geq 0$, where $\tau^{-m} \coloneqq \tau^{-1} \circ \dots \circ \tau^{-1}$ is a composition of $m$ terms and similarly for $\tau^m$. Importantly, the image of the intervals under $\overline{(-)}: \tors A \to \tors B$ in the top row is the trivial interval $[\Fac_{B} (\begin{smallmatrix}1\\2\end{smallmatrix}),\Fac_{B}(\begin{smallmatrix}1\\2\end{smallmatrix})]$ and the image of the intervals in the bottom row is the trivial interval $[\Fac_{B} (\begin{smallmatrix}1\end{smallmatrix}), \Fac_{B} (\begin{smallmatrix}1\end{smallmatrix})]$. Hence no $\tau$-perpendicular interval of $\tors A$ maps onto the interval $[ \Fac_{B} (\begin{smallmatrix}1\end{smallmatrix}), \Fac_{B} (\begin{smallmatrix}1\\2\end{smallmatrix})] \subseteq \tors B$. In particular, the module $\begin{smallmatrix}1\\2\end{smallmatrix}$ is support $\tau$-tilting in $\mods B$ but not in $\mods A$, where it is a regular module. 
\end{exmp}

The example \cite[Exmp. 5.11]{DIRRT2017} illustrates that the restriction $\overline{(-)}: \ftors A \to \ftors A/I$ is also not a surjection when both $A$ and $A/I$ are $\tau$-tilting infinite.  \\

Beside the existence of a faithful group functor, the pairwise compatibility property of morphisms in the $\tau$-cluster morphism category is a sufficient condition for $\Bcal \Tfrak(A/I)$ to be a $K(\pi,1)$ space \cite{Igusa2022}. This condition is related to the study of 2-simple minded collections \cite[Lem. 2.16]{HansonIgusa2021}. It is known that not all algebras satisfy the pairwise compatibility of last factors, see \cite[Thm. 4.1]{HansonIgusaPW2SMC} and \cite[Thm. 4]{BarnardHanson2022}. The following example illustrates that the pairwise compatibility property is independent of taking quotient algebras.

\begin{exmp} \label{exmp:DynkinD4}
    Let $A = KD_4$ be the representation-finite hereditary algebra of Dynkin type $D_4$ with orientation 
    \[ D_4 : \begin{tikzcd} 1 \arrow[rd, "\gamma_1"] \\
    & 2 \arrow[r, "\gamma_2"] & 3 \\
    4 \arrow[ru, "\gamma_4", swap] \end{tikzcd} \]

It was shown in \cite[Thm. 2.5]{IgusaTodorov22} that $\Tfrak(A)$ satisfies the pairwise compatibility property since $A$ is hereditary of finite type. The quotient $A/ \langle \gamma_4 \gamma_2 \rangle$ is gentle with no loops and 2-cycles, thus \cite[Thm. 4.1]{HansonIgusaPW2SMC} implies that $A/I$ does not have the pairwise compatibility condition since there exists a vertex of valency greater than 2. Hence taking quotients does not preserve the pairwise compatibility condition. Moreover, $A$ is a quotient of the preprojective algebra $\Pi_{D_4}$ of type $D_4$ which was shown in \cite[Thm. 4]{BarnardHanson2022} to not have the pairwise compatibility property, so taking quotient also does not preserve the failure of the pairwise compatibility property. In total we have a sequence of algebra epimorphisms:
\[ \Pi_{D_4} \twoheadrightarrow D_4 \twoheadrightarrow D_4/ \langle \gamma_4 \gamma_2 \rangle \twoheadrightarrow K^4 \]
where $K^4$ is the semisimple algebra on 4 vertices, which does satisfy the pairwise compatibility property.
\end{exmp}

\subsection*{Acknowledgements} The author is grateful to Sibylle Schroll and Hipolito Treffinger for their guidance and many meaningful conversations and suggestions. The author also thanks Erlend D. Børve and Calvin Pfeifer for inspiring discussions.

\subsection*{Funding} 
MK is supported by the Deutsche Forschungsgemeinschaft (DFG, German Research Foundation) -- Project-ID 281071066 -- TRR 191.

\newcommand{\etalchar}[1]{$^{#1}$}

\end{document}